\documentclass[12pt,a4paper,oneside,reqno,final]{amsart}
\usepackage{hyperref} 
\usepackage[inline]{showlabels} 
\showlabels{bibitem} 

\usepackage{amsthm, amsmath, amssymb,amscd}
\usepackage[T1]{fontenc}
\usepackage[utf8]{inputenc}
\usepackage[english]{babel}
\usepackage{typearea}
\usepackage{bookmark}
\usepackage{yfonts}
\usepackage{mathrsfs}
\usepackage{hyperref}
\usepackage{verbatim}
\usepackage{color}
\usepackage{float}

\newcommand{\LRA}{\longrightarrow}

\newcommand{\N}{\mathbb{N}} 
\newcommand{\Z}{\mathbb{Z}}

\newcommand{\D}{\mathcal{D}} 
 
\newcommand{\K}{\mathbb{K}} 
\newcommand{\E}{\mathcal{E}} 
\newcommand{\F}{\mathcal{F}}

\DeclareMathOperator{\id}{id}

{}
\let\leq\leqslant{} 
\let\geq\geqslant{}
\let\subset\subseteq{}  
{}

\usepackage{tikz-cd}	
\usetikzlibrary{matrix,arrows,decorations.markings}
\usepgflibrary{arrows.meta}

\newtheorem{lemma}{Lemma}[section] 
\newtheorem{corollary}[lemma]{Corollary}
\newtheorem{theorem}[lemma]{Theorem}
\newtheorem*{theorem*}{Theorem}
\newtheorem{proposition}[lemma]{Proposition}

\theoremstyle{definition}
\newtheorem{definition}[lemma]{Definition}

\newtheorem{example}[lemma]{Example}

\newtheorem{remark}[lemma]{Remark}

\numberwithin{equation}{section} 

\title[Eventual conjugacy]{Balanced strong shift equivalence, balanced in-splits, and eventual conjugacy}

\author[K.A. Brix]{Kevin Aguyar Brix}
\address{School of Mathematics and Applied Statistics, University of Wollongong, Australia}
\email{kabrix.math@fastmail.com}

\makeatletter
\@namedef{subjclassname@2020}{%
  \textup{2020} Mathematics Subject Classification}
\makeatother

\begin{document}

\begin{abstract}
    We introduce the notion of balanced strong shift equivalence between square nonnegative integer matrices, 
    and show that two finite graphs with no sinks are one-sided eventually conjugate if and only if their adjacency matrices are conjugate to balanced strong shift equivalent matrices.
    Moreover, we show that such graphs are eventually conjugate if and only if one can be reached by the other via a sequence of out-splits and balanced in-splits;
    the latter move being a variation of the classical in-split move introduced by Williams in his study of shifts of finite type.
    We also relate one-sided eventual conjugacies to certain block maps on the finite paths of the graphs.
    These characterizations emphasize that eventual conjugacy is the one-sided analog of two-sided conjugacy.
\end{abstract}
\maketitle


\section{Introduction}
Shifts of finite type are accurately represented as finite essential directed graphs with associated adjacency matrices,
and Williams' seminal paper~\cite{Williams1973} was a deep investigation into two-sided conjugacy of shifts of finite type.
He successfully characterized conjugacy in terms of an equivalence relation on nonnegative integer matrices, \emph{strong shift equivalence}:
a pair of nonnegative integer matrices \textbf{A} and \textbf{B} are elementary equivalent if there is a pair of rectangular nonnegative matrices $(\textbf{R}, \textbf{S})$ satisfying
\[
    \textbf{A} = \textbf{R} \textbf{S}, \quad \textbf{S} \textbf{R} = \textbf{B},
\]
and strong shift equivalence is the transitive closure of elementary equivalence.
Two shifts of finite type are then two-sided conjugate if and only if their corresponding matrices are strong shift equivalent.
This algebraic point of view is intimately related to the notion of shift equivalence (though the two relations are distinct by work of Kim and Roush~\cite{Kim-Roush}).
While shift equivalence is decidable~\cite{Kim-Roush1988}, it is arguably the biggest open problem in symbolic dynamics to determine whether strong shift equivalence is a decidable relation.

An important tool in Williams' work, and one which is central to this paper, is that of state splittings.
In the coordinate-free approach to symbolic dynamics, shifts of finite type are maps on compact metric spaces satisfying a certain Markov condition with respect to a partition of the space.
A state splitting then amounts to refining the partition in a way which is compatible with the map.
Such a procedure produces different matrix representations of the same map which are strong shift equivalent.
At the level of graphs, this amounts to splitting vertices and appropriately distributing the out-going edges (out-splits) or the in-going edges (in-splits).
Williams then showed that any conjugacy can be decomposed into a finite sequence of state splittings.
One formidable consequence was the classification of one-sided shifts of finite type in terms of only out-splits.

Motivated by equivalences of graph $\mathrm{C^*}$-algebras, Bates and Pask~\cite{Bates-Pask2004} generalized these graphical constructions of state splittings
--- which we now call \emph{moves} --- 
to arbitrary directed graphs and showed that whereas out-splits (Move \texttt{(O)}) produce $^*$-isomorphic graph $\mathrm{C^*}$-algebras,
the in-split (Move \texttt{(I)}) only preserves the Morita equivalence class of the graph $\mathrm{C^*}$-algebras.
The coarser relation of \emph{flow equivalence} of the dynamical systems was shown by Parry and Sullivan~\cite{Parry-Sullivan} to be generated by conjugacies and symbol expansions.
This expansion or stretching of time corresponds to a delay move on the graphs which again preserves the Morita equivalence class of the $\mathrm{C^*}$-algebra.
In fact, it was known since the inception of the graph $\mathrm{C^*}$-algebras of finite essential graphs (irreducible and not just a cycle) 
--- the \emph{Cuntz--Krieger algebras} ---
that the stabilized graph $\mathrm{C^*}$-algebra together with its canonical diagonal subalgebra is invariant under flow equivalence of the underlying dynamical systems~\cite{CK80}.

Only much later with a remarkable use of groupoid techniques did Matsumoto and Matui~\cite{MM14} finally prove the converse statement: 
The stable isomorphism class of simple Cuntz--Krieger algebras \emph{together with their diagonal subalgebra} (suitably stabilized)
completely remembers the flow class of the underlying shift of finite type.
It follows that the stable Cuntz--Krieger algebras are diagonally $^*$-isomorphic if and only if the underlying graphs can be connected by a finite sequence of splittings and stretchings.
This was later generalized to all finite essential graphs~\cite{CEOR}.

Moreover, Carlsen and Rout~\cite{Carlsen-Rout} showed that two finite essential graphs $E$ and $F$ are two-sided conjugate if and only if
there is a $^*$-isomorphism $\Psi\colon \mathrm{C^*}(E)\otimes \K \LRA \mathrm{C^*}(F)\otimes \K$ which respects the diagonal subalgebras
$\Psi(\D(E)\otimes c_0) = \D(F)\otimes c_0$ and intertwines the gauge actions $\Psi\circ (\gamma^E\otimes \id) = (\gamma^F\otimes \id)\circ \Psi$.
Here, $\K$ is the $\mathrm{C^*}$-algebra of compact operators on separable Hilbert space and $c_0$ is the $\mathrm{C^*}$-subalgebra of diagonal operators.
Hence this kind of equivalence of graph $\mathrm{C^*}$-algebras is implemented by out-splits and in-splits of the underlying graphs by Williams' result.
A general program for determining which moves on the graphs captures structure-preserving $^*$-isomorphisms of the graph $\mathrm{C^*}$-algebras 
was recently initiated by Eilers and Ruiz~\cite{Eilers-Ruiz}.

The present paper addresses one-sided eventual conjugacy as studied by Matsumoto and its relation to two-sided conjugacy.
The motivation is three-fold.

Firstly, work of Matsumoto~\cite{Mat2017-circle,Mat17-UCOE} and Carlsen and Rout~\cite{Carlsen-Rout} showed that a pair of graphs $E$ and $F$ are \emph{one-sided eventually conjugate} if and only if
there is a $^*$-isomorphism $\Psi\colon \mathrm{C^*}(E) \LRA \mathrm{C^*}(F)$ which respects the diagonal $\Psi(\D(E)) = \D(F)$
and intertwines the gauge actions $\Psi\circ \gamma^E = \gamma^F\circ \Psi$.
This leads us to the observation that one-sided eventual conjugacy is \emph{structurally related} to two-sided conjugacy.
We may also say that eventual conjugacy is the one-sided analog of two-sided conjugacy.
\emph{A fortiori}, one-sided conjugacy is \emph{not} the one-sided analog of two-sided conjugacy in this sense, cf. work of the author and Carlsen~\cite{BC}.

From a dynamical perspective, two-sided conjugacies are sliding block codes with a finite ``window'' (with memory and anticipation)
arising from block maps on finite paths of the graphs.
Similarly, one-sided conjugacies are sliding block codes with anticipation but no memory.
It is possible to make sense of one-sided sliding block codes with both anticipation and memory 
and, in fact, such sliding block codes induce all eventual conjugacies.

\begin{theorem*}[Corollary~\ref{cor:1:1}]
    Let $E$ and $F$ be finite graphs with no sinks.
    There is a one-to-one correspondence between eventual conjugacies $h\colon E^\infty \LRA F^\infty$ and
    $(\ell, c)$-block maps $\psi\colon E^{1 + \ell + c} \LRA F^{1 + \ell}$ satisfying a bijectivity condition.
\end{theorem*}

The connection between two-sided conjugacy and one-sided eventual conjugacy it seems
is present before passing to infinite sequences.

The second motivation is connected to Eilers and Ruiz' recent program of understanding the relationship between moves on graphs 
and structure-preserving $^*$-isomorphisms of graph $\mathrm{C^*}$-algebras~\cite{Eilers-Ruiz}.
We study their variation of Williams' classical in-split called the \emph{balanced in-split} (Move (\texttt{I+})) and show that it induces eventually conjugate graphs.
This new move introduces sources in the graphs which is why we consider finite graphs with no sinks but \emph{potentially with sources}.
We arrive at the following result.

\begin{theorem*}[Theorem~\ref{thm:generated}]
    One-sided eventual conjugacy of finite graphs with no sinks is generated by out-splits and balanced in-splits.    
\end{theorem*}

The presence of sources is an essential part of the proof even if the graphs we start with do not have sources.
It also follows from the proof that one-sided conjugacy is generated by out-splits alone, cf.~\cite{Williams1973, Boyle-Franks-Kitchens, Kitchens}.
Since work of Carlsen and Rout~\cite{Carlsen-Rout} shows that 
eventual conjugacy of the graphs is equivalent to diagonal-preserving $^*$-isomorphism of the graph algebras which intertwines the gauge actions
--- this is a \texttt{111}-isomorphism in the terminology of~\cite{Eilers-Ruiz} ---
the result characterizes this particular structure-preserving $^*$-isomorphism in terms of Move (\texttt{O}) and Move (\texttt{I+}).

Thirdly, we introduce \emph{balanced strong shift equivalence} of nonnegative integer matrices which is akin to strong shift equivalence.
A pair of matrices \textbf{A} and \textbf{B} are balanced elementary equivalent if there is a triple of rectangular nonnegative integer matrices $(\textbf{R}_\textbf{A}, \textbf{S}, \textbf{R}_\textbf{B})$
satisfying
\[
    \textbf{A} = \textbf{S} \textbf{R}_\textbf{A}, \quad \textbf{R}_\textbf{A} \textbf{S} = \textbf{R}_\textbf{B} \textbf{S} \quad \textbf{B} = \textbf{S} \textbf{R}_\textbf{B},
\]
and balanced strong shift equivalence is the transitive closure of balanced elementary equivalence.
It is an interesting question whether balanced strong shift equivalence is a decidable relation.
As an analog of Williams' algebraic description of conjugacy, we find the following result.

\begin{theorem*}[Theorem~\ref{thm:balanced-strong-shift-equivalence}]
    A pair of finite graphs $E$ and $F$ with no sinks are eventually conjugate if and only if they are conjugate to graphs $E'$ and $F'$, respectively,
    whose adjacency matrices are balanced strong shift equivalent.
\end{theorem*}

Even if the present work is in part motivated by operator algebraic questions, the approach is purely dynamical.
Hopefully, a thorough investigation of one-sided eventual conjugacy can also shed light on questions concerning two-sided conjugacy as well.

\section*{Acknowledgements}
This work was carried out while the author was a PhD student at the University of Copenhagen
and supported by the Danish National Research Foundation through the Centre for Symmetry and Deformation (DNRF92)
and finished at the University of Wollongong during a Carlsberg Foundation Internationalisation Fellowship.
A previous version of this paper appeared in the author's PhD thesis.
I would like to thank Søren Eilers and Efren Ruiz for valuable comments on an earlier draft of this paper,
and Aidan Sims and the University of Wollongong for their hospitality during my visit.

\section{Preliminaries}

We recall the relevant concepts of finite graphs and the dynamics on their path spaces and fix notation.
Let $\Z$,  $\N = \{0, 1, 2,\ldots\}$ and $\N_+ = \{1, 2,\ldots\}$ denote the integers, the nonnegative integers and the positive integers, respectively.

A \emph{directed graph} (or just a graph) is a quadrupel $E = (E^0, E^1, r_E, s_E)$ where $E^0$ is the set of \emph{vertices},
$E^1$ is the set of \emph{edges} and $s_E, r_E\colon E^1\LRA E^0$ are the \emph{source} and \emph{range} maps, respectively.
We shall omit the subscripts to simplify notation when the graph is understood.
A vertex $v\in E^0$ is a \emph{source} if $r^{-1}(v) = \emptyset$ and a \emph{sink} if $s^{-1}(v) = \emptyset$. \\

\textbf{Assumption.} \emph{In this paper we consider only finite graphs with no sinks}. \\

A \emph{finite path} is a string $\mu = \mu_1 \cdots \mu_n$ of edges
where $n\in \N_+$ and $r(\mu_i) = s(\mu_{i+1})$ for $i = 1,\ldots n-1$.
The \emph{length} of $\mu$ is $|\mu| = n$.
By convention, the length of a vertex is zero.
Let $E^n$ be all finite paths of length $n$ and let $E^* = \bigcup_{n \in \N} E^n$ be the collection of all finite paths.
The source and range maps naturally extend to $E^*$ by $s(\mu) = s(\mu_1)$ and $r(\mu) = r(\mu_{|\mu|})$.
The \emph{path space} of $E$ is the set
\[
    E^\infty = \{ x \in {(E^1)}^\N \mid  r(x_i) = s(x_{i + 1}), i\in \N\}
\]
of all infinite paths of edges on the graph.
The path space $E^\infty$ is compact and Hausdorff in the subspace topology of the product topology on ${(E^1)}^\N$ where $E^1$ is discrete.
The source map extends to $E^\infty$ by putting $s(x) = s(x_0)$ for $x\in E^\infty$.
Given $x\in E^\infty$ we write $x_{[i, j]} = x_i x_{i + 1} \cdots x_j$ for $0 \leq i \leq j$, and put
$x_{[i, j)} = x_{[i, j - 1]}$ and $x_{(i, j]} = x_{[i + 1, j]}$ when $i < j$.
Note that any finite path $\mu \in E^*$ is of the form $x_{[i, i + |\mu|)}$ for some $x\in E^\infty$ and $i\in \N$.
The \emph{cylinder set} of a finite path $\mu$ is the compact open set
\[
    Z(\mu) = \{ x\in E^\infty \mid x_{[0, |\mu|)} = \mu \},
\]
and the collection of cylinder sets constitutes a basis for the topology of $E^\infty$.

Define a \emph{shift operation} $\sigma_E\colon E^\infty \LRA E^\infty$ by $\sigma_E(x)_i = x_{i + 1}$ for $x\in E^\infty$ and $i\in \N$.
This is a local homeomorphism, and it is surjective if and only if $E$ contains no sources.
The dynamical system $(E^\infty, \sigma_E)$ is the \emph{edge shift} of $E$.
Two edge shifts $(E^\infty, \sigma_E)$ and $(F^\infty, \sigma_F)$ are \emph{conjugate} if there exists a homeomorphism $h\colon E^\infty \LRA F^\infty$
such that $h\circ \sigma_E = \sigma_F\circ h$.
In this case we say the graphs $E$ and $F$ are conjugate and that $h$ is a conjugacy.
Edge shifts are examples of shifts of finite type and, in fact, every shift of finite type is (conjugate to) an edge shift of a finite graph~\cite[Section 2]{LM}.

Let $E$ be a graph.
The \emph{$N$'th higher block graph} of $E$ is the graph $E^{[N]}$ with edge set $E^N$ and vertex set $E^{N - 1}$.
The source and range of $\mu = \mu_1 \cdots \mu_N\in E^N$ is $s_N(\mu) = \mu_{[1, N)}$ and $r_N(\mu) = \mu_{(1, N]}$.
Note that $E^{[1]} = E$, and that for any $N$ there is a canonical conjugacy $\varphi\colon E^\infty \LRA {(E^{[N]})}^\infty$ 
given by $\varphi(x) = x_{[0, N)} x_{[1, N + 1)} x_{[2, N + 2)}\cdots$ for $x\in E^\infty$.

\section{Eventual conjugacy and block maps}
\label{sec:block-maps}

Let us fix two finite graphs with no sinks $E$ and $F$ and let $E^\infty$ and $F^\infty$ be their infinite path spaces, respectively.

\subsection{Eventual conjugacy}

A \emph{sliding block code} $\varphi\colon E^\infty \LRA F^\infty$ is a continuous map which intertwines the shift operations
in the sense that $\varphi\circ \sigma_E = \sigma_F\circ \varphi$.
A \emph{conjugacy} is a bijective sliding block code.
In order to describe eventual conjugacies on graphs we first modify the notion of sliding block codes slightly.

\begin{definition}
    Let $\ell\in \N$.
    An $\ell$-\emph{sliding block code} is a continuous map $\varphi\colon E^{\infty}\LRA F^\infty$ such that $\sigma_E^\ell\circ \varphi$
    is a sliding block code, that is,
    \[
        \sigma_F^{\ell + 1}(\varphi(x)) = \sigma_F^\ell(\varphi(\sigma_E(x))),
    \]
    for every $x\in E^{\infty}$.
    If $\ell = 0$, then $\varphi$ is a sliding block code in the usual sense.
    An \emph{$\ell$-conjugacy} is a bijective $\ell$-sliding block code.
\end{definition}

An $\ell$-conjugacy is simply one half of an eventual conjugacy.

\begin{definition}[\cite{Mat2017-circle}]
    Two finite graphs with no sinks $E$ and $F$ are \emph{eventually conjugate} if there exist a homeomorphism $h\colon E^\infty\LRA F^\infty$ 
    and $\ell, \ell'\in \N$ such that $h$ and $h^{-1}$ are $\ell$- and $\ell'$-conjugacies, respectively.
    That is, 
    \begin{align*}
        \sigma_F^{\ell + 1}(h(x)) &= \sigma_F^\ell(h(\sigma_E(x))), \\
        \sigma_E^{\ell' + 1}(h^{-1}(y)) &= \sigma_E^{\ell'}(h^{-1}(\sigma_F(y)))     ,
    \end{align*}
    for $x\in E^\infty$ and $y\in F^\infty$.
    We say that $h$ is an \emph{eventual conjugacy}.
\end{definition}

If $h\colon E^\infty \LRA F^{\infty}$ is an $\ell$-conjugacy, then $h$ is an $(\ell + i)$-conjugacy for all $i\in \N$, and
if $h'\colon F^\infty\LRA G^\infty$ is an $\ell'$-conjugacy, then $h'\circ h$ is an $(\ell + \ell')$-conjugacy.
Furthermore, if $h$ is an eventual conjugacy, then there is an $\ell\in \N$ such that $h$ and $h^{-1}$ are $\ell$-conjugacies.

Let $\ell\in \N$ and suppose $h\colon E^\infty\LRA F^\infty$ is an $\ell$-sliding block code.
Then there is a $c\in \N$ such that
\begin{equation}\label{eq:continuity-constant}
    h(Z_E(x_{[0, k + c]})) \subset Z_F({h(x)}_{[0, k]})
\end{equation}
for $k\geq \ell$ and $x\in E^\infty$.
We shall refer to $c$ as a \emph{continuity constant for $h$ (relative to $\ell$)}.   

Note that if $c$ is a continuity constant for $h$ (relative to $\ell$), then $c + i$ is a continuity constant for $h$ for any $i\in \N$.
Moreover, if $h$ and $h^{-1}$ are $\ell$- and $\ell'$-conjugacies, then we can choose a continuity constant relative to $\max\{\ell, \ell'\}$.

\begin{definition}
    An \emph{$(\ell, c)$-sliding block code} is an $\ell$-sliding block code $\varphi\colon E^\infty\LRA F^\infty$ with continuity constant $c$ relative to $\ell$.
    An \emph{$(\ell, c)$-conjugacy} is a bijective $(\ell, c)$-sliding block code.
\end{definition}

The following lemma shows that we can always reduce the continuity complexity of an $\ell$-conjugacy,
at the expense of increasing the continuity complexity of the inverse.

\begin{lemma}
    Let $E$ and $F$ be finite graphs with no sinks and let $h\colon E^\infty\LRA F^\infty$ be an $(\ell, c)$-conjugacy.
    There exists a graph $\bar{E}$ with a conjugacy $\varphi\colon E^\infty \LRA {\bar{E}}^\infty$
    and an $(\ell, 0)$-conjugacy $\bar{h}\colon \bar{E}^\infty\LRA F^\infty$ satisfying $\bar{h}\circ \varphi = h$.
\end{lemma}

\begin{proof}
    Consider the higher block graph $\bar{E} = E^{[c]}$ and let $\varphi\colon E^\infty\LRA {\bar{E}}^\infty$ be the canonical conjugacy.
    Define $\bar{h}\colon \bar{E}^\infty\LRA F^\infty$ as $\bar{h} = h\circ \varphi^{-1}$, that is,
    \[
        \bar{h}(x_{[0, c]} x_{[1, c + 1]} \cdots) = h(x),
    \]
    for $x\in E^\infty$.
    Then $\bar{h}$ is an $\ell$-conjugacy.
    Since $\bar{h} = h\circ \varphi^{-1}$ and $c$ is a continuity constant for $h$, we have 
    \[
        \bar{h}(Z_{\bar{E}}(\bar{x}_{[0, \ell]})) = h(Z_E(x_{[0, \ell + c]})) \subset Z_F(h(x)_{[0, \ell]})
    \]
    for $\bar{x}\in {\bar{E}}^\infty$ with $\varphi^{-1}(\bar{x}) = x\in E^\infty$.
    Hence $\bar{h}$ is an $(\ell, 0)$-conjugacy.
\end{proof}

\subsection{Block maps}

\begin{definition}
    Let $E$ and $F$ be finite graphs with no sinks and let $\ell, c\in \N$.
    An \emph{$(\ell, c)$-block map} is a map $\psi\colon E^{1 + \ell + c}\LRA F^{1 + \ell}$ which is \emph{compatible with $E$ and $F$}
    in the sense that 
    \[
        {\psi(x_{[0, \ell + c]})}_\ell {\psi(x_{[1, 1 + \ell + c]})}_{\ell}\in F^2,
    \]
    for $x\in E^{\infty}$.
\end{definition}

Let $\psi\colon E^{1 + \ell + c}\LRA F^{1 + \ell}$ be an $(\ell, c)$-block map.
The compatibility condition ensures that there is an extension $\psi\colon E^{1 + \ell + c + i} \LRA F^{1 + \ell + i}$ given by
\[
  \psi(x_{[0, \ell + c + i]}) = \psi(x_{[0, \ell + c]}) {\psi(x_{[1, \ell + c + 1]})}_\ell \cdots {\psi(x_{[i, \ell + c + i]})}_\ell,
\]
whenever $x\in E^\infty$ and $i\in \N$.
We shall use the same symbol $\psi$ for the block map and its extension, this should cause no confusion.
Iterating this process \emph{ad infinitum}, $\psi$ extends to a map on the infinite path spaces $h = h_\psi\colon E^\infty\LRA F^\infty$ by
\[
    h(x) = \psi(x_{[0, \ell + c]}) {\psi(x_{[1, \ell + c + 1]})}_\ell \cdots {\psi(x_{[i, \ell + c + i]})}_\ell \cdots,
\]
for $x\in E^\infty$.
In order to see that $h$ is continuous, suppose $x^{(n)}\LRA x$ in $E^\infty$ as $n\LRA \infty$.
Given $i\in \N$ there exists $N\in \N$ such that $x^{(n)}\in Z_E(x_{[0, \ell + c + i]})$ whenever $n\geq N$.
Then
\[
    {h(x^{(n)})}_{[0, \ell + i]} = \psi({x^{(n)}}_{[0, \ell + c + i]}) = \psi(x_{[0, \ell + c + i]}) = {h(x)}_{[0, \ell + i]},
\]
so $h(x^{(n)}) \LRA h(x)$ in $F^\infty$ as $n\LRA \infty$.
In particular, $c$ is a continuity constant for $h$ relative to $\ell$.
We would now like to find conditions on the block maps which ensure that the induced map on the path spaces is bijective.

\begin{definition}
    An $(\ell, c)$-block map $\psi\colon E^{1 + \ell + c}\LRA F^{1 + \ell}$ satisfies the \emph{surjectivity condition}
    if for every $k\geq \ell$ and every $\beta\in F^{1 + k}$, there exists $\alpha\in E^{1 + k + c}$ such that $\psi(\alpha) = \beta$.
    On the other hand, $\psi$ satisfies the \emph{injectivity condition} if for every $k\geq \ell$ there is $K\in \N$ such that
    \[
      \psi(x_{[0, k + K]}) = \psi(x'_{[0, k + K]}) \implies x_{[0, k]} = x'_{[0, k]},
    \]
    for every $x, x'\in E^\infty$.
    We say that $\psi$ satisfies the \emph{bijectivity condition} if it satisfies both the injectivity and the surjectivity conditions.
\end{definition}

\begin{remark}
    Note that for the case $c = 0$, the surjectivity condition reduces to surjectivity of the block map
    while the injectivity condition is, in general, weaker than injectivity of the block map.
    The term \emph{bijectivity condition} is only meant to reflect the fact that the induced map on the infinite path spaces is bijective (see the proof below).
    If $\psi\colon E^{1 + \ell} \LRA F^{1 + \ell}$ is in fact a \emph{bijective} block map,
    then the inverse homeomorphism ${h_\psi}^{-1}\colon E^\infty \LRA F^\infty$ is induced by $\psi^{-1}$
    and has vanishing continuity constant.
\end{remark}

We now arrive at the main result of this section.

\begin{theorem}\label{thm:1:1}
    Let $E$ and $F$ be finite graphs with no sinks and let $\ell, c\in \N$.
    There is a one-to-one correspondence between the collection of $(\ell, c)$-block maps $\psi\colon E^{1 + \ell + c}\LRA F^{1 + \ell}$ 
    and the collection of $(\ell, c)$-sliding block codes $h\colon E^\infty\LRA F^\infty$.
    In addition, $\psi$ satisfies the injectivitiy condition or the surjectivity condition
    if and only if $h$ is injective or surjective, respectively.
\end{theorem}

\begin{proof}
    Suppose $\psi\colon E^{1 + \ell + c}\LRA F^{1 + \ell}$ is an $(\ell, c)$-block map.
    Let $h = h_\psi\colon E^\infty \LRA F^\infty$ be the continuous extension to the infinite path spaces given by
    \[
      {h(x)}_{[0, k]} = \psi(x_{[0, k + c]}),
    \]
    for $x\in E^\infty$ and $k\geq \ell$.
    Note that $c$ is a continuity constant for $h$ relative to $\ell$.
    If $ax\in E^\infty$, then
    \begin{align*}
        \sigma_F^{\ell + 1}(h(ax)) 
        &= \sigma_F^{\ell + 1}
        \Big( 
        {\psi( {(ax)}_{[0, \ell + c]})} {\psi( {(ax)}_{[1, \ell + c + 1]})}_{\ell} \cdots {\psi( {(ax)}_{[i, \ell + c + i]})}_{\ell} \cdots 
        \Big) \\
        &= {\psi( {(ax)}_{[1, 1 + \ell + c]})}_{\ell} \cdots {\psi( {(ax)}_{[i, \ell + c + i]})}_{\ell}\cdots \\
        &= {\psi( x_{[0, \ell + c]})}_{\ell} \cdots {\psi( x_{[i, \ell + c + i]})}_{\ell}\cdots \\
        &= \sigma_F^\ell
        \Big( 
        \psi(x_{[0, \ell + c]}) \cdots {\psi( x_{[i, \ell + c + i]})}_{\ell}\cdots 
        \Big) \\
        &= \sigma_F^\ell(h(x)).
    \end{align*}
    Hence $h$ is an $(\ell, c)$-sliding block code.

Suppose $\psi$ satisfies the surjectivity condition and let $y\in F^\infty$.
For each $k\geq \ell$ choose $\alpha^{(k)}\in E^{1 + k + c}$ such that $\psi(\alpha^{(k)}) = y_{[0, k]}$.
Pick $x^{(k)}\in Z_E(\alpha^{(k)})$.
Since $E^\infty$ is compact the sequence ${(x^{(k)})}_k$ has a convergent subsequence;
let $x\in E^\infty$ be its limit. 
Then $h(x) = y$ and $h$ is surjective.

Next, assume $\psi$ satisfies the injectivity condition and suppose $h(x) = h(x')$ for some $x, x'\in E^\infty$.
Let $k\geq \ell$ and choose $K\geq c$ in accordance with the injectivity condition.
Then 
\[
    \psi(x_{[0, k + K]})  = {h(x)}_{[0, k + K - c]} = {h(x')}_{[0, k + K - c]}
    = \psi(x'_{[0, k + K]}),
\]
from which it follows that $x_{[0, k]} = x'_{[0, k]}$.
As this is the case for every $k\geq \ell$, we see that $x = x'$ and that $h$ is injective.

For the other direction, let $h\colon E^\infty\LRA F^\infty$ be an $(\ell, c)$-sliding block code.
 Define $\psi\colon E^{1 + \ell + c}\LRA F^{1 + \ell}$ by
\[
  \psi(x_{[0, \ell + c]}) = {h(x)}_{[0, \ell]},
\]
for $x\in E^{\infty}$.
This is well-defined by the choice of $c$ and $\psi$ is compatible with $E$ and $F$.
Hence $\psi$ is an $(\ell, c)$-block map.
Furthermore, the $(\ell, c)$-sliding block code $h_\psi$ induced by $\psi$ coincides with $h$.

Fix $\beta\in F^{1 + k}$ for some $k \geq \ell$ and let $\beta y\in Z_F(\beta)$.
If $h$ is surjective, we may choose $x\in E^\infty$ such that $h(x) = \beta y$.
Note that $\psi(x_{[0, k + c]}) = {h(x)}_{[0, k]} = \beta$.
Hence $\psi$ satisfies the surjectivity condition.

If $h$ is injective, then it is a homeomorphism onto its image;
let $\phi\colon h(E^\infty) \LRA E^\infty$ be the inverse homeomorphism.
Fix $k\geq \ell$.
Since $\phi$ is continuous there is $c_\phi\in \N$ such that $\phi({h(x)}_{[0, k + c_\phi]}) = \phi({h(x')}_{[0, k + c_\phi]})$ implies $x_{[0, k]} = x'_{[0, k]}$ for $x, x'\in E^\infty$.
If $\psi(x_{[0, k + c + c_\phi]}) = \psi(x'_{[0, k + c + c_\phi]})$, for some $x, x'\in E^\infty$,
then
\[
    {h(x)}_{[0, k + c_\phi]} = {h(x')}_{[0, k + c_\phi]}
\]
by the choice of $c$ and      
\[
    x_{[0, k]} = {(\phi h(x))}_{[0, k]} = {(\phi h(x'))}_{[0, k]} = x'_{[0, k]}
\]
by the choice of $c_\phi$.
Hence $\psi$ satisfies the injectivity condition.

Finally, it is straightforward to verify that if $\psi\colon E^{1 + \ell + c}\LRA F^{1 + \ell}$ is an $(\ell, c)$-block map
and $h$ is the $(\ell, c)$-sliding block code induced from $\psi$, then the $(\ell, c)$-block map $\psi_h$ induced from $h$ coincides with $\psi$.
\end{proof}

We record the following immediate corollary.

\begin{corollary}\label{cor:1:1}
    Let $\ell, c\in \N$.
    There is a one-to-one correspondence between the collection of $(\ell, c)$-block maps $\psi\colon E^{1 + \ell + c}\LRA F^{1 + \ell}$ 
    satisfying the \emph{bijectivity condition} and the collection of $(\ell, c)$-conjugacies $h\colon E^\infty\LRA F^\infty$.
    For $c = 0$, there is a one-to-one correspondence between the collection of \emph{bijective} block maps $\psi\colon E^{1 + \ell}\LRA F^{1 + \ell}$ 
    and the collection of $(\ell,0)$-conjugacies $h\colon E^\infty \LRA F^\infty$ whose inverses are also $(\ell, 0)$-conjugacies.
\end{corollary}

\begin{remark}
    At the level of block maps, the index $(\ell, c)$ can be interpreted as memory and anticipation, respectively. 
    For one-sided conjugacies, \emph{memory is not allowed} (cf.~\cite[Section 13.8]{LM}), however, we have seen that memory of the block map
    exactly corresponds to eventual conjugacies.
    At the level of path spaces, the $\ell$ indicates the lag of the eventual conjugacy while $c$ indicates the continuity complexity of the associated homeomorphism.
\end{remark}

\section{Moves on graphs}
\label{moves}

In this section, we relate eventual conjugacies with moves on the graph.
We first recall the definition of the out-split and in-split of~\cite{Williams1973} modified for graphs (cf.~\cite{LM, Bates-Pask2004}).
For simplicity, we only define these moves for finite graphs with no sinks.

\begin{definition}[Move \texttt{(O)}]
   Let $G = (G^0, G^1, r_G, s_G)$ be a finite graph with no sinks.  
   Let $v\in G^0$ be a vertex, let $n\in \N_+$ and partition $s_G^{-1}(v)$ into finitely many nonempty sets
   \begin{align} \label{eq:out-partition}
       s_G^{-1}(v) = \E_v^1 \amalg \cdots \amalg\E_v^n.
   \end{align}
   The \emph{out-split graph} $E$ of $G$ at $v$ with respect to the partition~\eqref{eq:out-partition} is given by
   \begin{align*}
       E^0 &= \{ v_1,\ldots,v_n \} \cup \{w_1 \mid w\in G^0\setminus \{v\}\}, \\
       E^1 &= \{ e_1,\ldots,e_n \mid r_G(e) = v\} \cup \{ f_1 \mid r_G(f) \neq v\},
   \end{align*}
   with source and range maps $r_E, s_E\colon E^1\LRA E^0$ given by
    \begin{align*}
        r_E(e_i) &= 
        \begin{cases}
            r_G(e)_1    & \textrm{if } s_G(e)\neq v, \\
            v_i    & \textrm{if } s_G(e) = v
        \end{cases}\\
        s_E(e_i) &=
        \begin{cases}
            s_G(e)_1    & \textrm{if}~s_G(e)\neq v, \\
            v_j         & \textrm{if}~s_G(e) = v, e\in \E_v^j,
        \end{cases}
    \end{align*}
    for $e_i\in E^1$.
\end{definition}

A directed graph is conjugate to its out-split graph, see~\cite[Corollary 6.2]{BCW2017}.

\begin{example}\label{ex:golden-mean}
  Let us consider an example of the out-splitting process.
  More examples can be found in~\cite[Section 2.4]{LM}.
 The graph
 
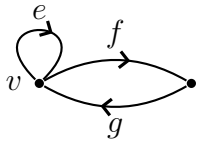
\begin{figure}[H]
    \begin{center}
    \begin{tikzpicture}
    [scale=5, ->-/.style={thick, decoration={markings, mark=at position 0.6 with {\arrow{Straight Barb[line width=0pt 1.5]}}},postaction={decorate}},
        node distance =2cm,
    thick,
    vertex/.style={inner sep=0pt, circle, fill=black}]
        \node(G) {};
        \node[vertex, right of = G,label=left:$v$] (G1) {.};
        \node[vertex, right of = G1] (G2) {.};

        \draw[->-, looseness=30, out=135, in=45] (G1) to node[above] {$e$} (G1);
        \draw[->-, bend left] (G1) to node[above] {$f$} (G2);
        \draw[->-, bend left] (G2) to node[below] {$g$} (G1);
   \end{tikzpicture}
    \end{center}
\end{figure}
\noindent encodes (a conjugate copy) of the golden mean shift, cf. e.g.~\cite[Example 1.2.3]{LM}.
The vertex $v$ emits two edges, so we can out-split it with respect to the partition 
\[
  s^{-1}(v) = \{e\} \amalg \{f\}.
\]
The resulting graph is
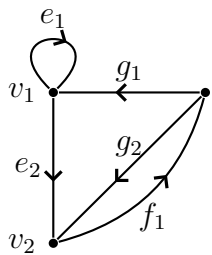
\begin{figure}[H]
    \begin{center}
    \begin{tikzpicture}
    [scale=5, ->-/.style={thick, decoration={markings, mark=at position 0.6 with {\arrow{Straight Barb[line width=0pt 1.5]}}},postaction={decorate}},
        node distance =2cm,
    thick,
    vertex/.style={inner sep=0pt, circle, fill=black}]
        \node(E) {};
        \node[vertex, right of = E, label=left:$v_1$] (E1) {.};
        \node[vertex, right of = E1] (E2) {.};
        \node[vertex, below of = E1, label=left:$v_2$] (E3) {.};

        \draw[->-, looseness=30, out=135, in=45] (E1) to node[above] {$e_1$} (E1);
        \draw[->-] (E1) to node[left] {$e_2$} (E3);
        \draw[->-, bend right] (E3) to node[below] {$f_1$} (E2);
        \draw[->-] (E2) to node[above] {$g_1$} (E1);
        \draw[->-] (E2) to node[above] {$g_2$} (E3);
   \end{tikzpicture}
    \end{center}
\end{figure}
\noindent and this is again a conjugate copy of the golden mean shift.
\end{example}

Next, we introduce a slight modification of the classical in-split graph, cf.~\cite{Eilers-Ruiz}.

\begin{definition}[Move \texttt{(I-)}]\label{def:(I-)}
    Let $G = (G^0, G^1, r_G, s_G)$ be a finite graph with no sinks.
    Let $v\in G^0$ be a vertex and partition $r_G^{-1}(v)$ into $n\in \N_+$ possibly empty sets
    \begin{align}\label{eq:in-partition}
        \E^v: r^{-1}(v) = \E^v_1 \amalg\cdots \amalg \E^v_n.
    \end{align}
    The \emph{in-split graph} $E$ of $G$ at $v$ with respect to the partition $\E$ is given by
    \begin{align*}
        E^0 &= \{ v^1,\ldots,v^n \} \cup \{w^1 \mid w\in G^0\setminus \{v\}\}, \\
        E^1 &= \{ e^1,\ldots, e^n \mid s_G(e) = v \} \cup \{f^1 \mid s_G(f) \neq v \}   
    \end{align*}
    with source and range maps $s_E, r_E\colon E^1 \LRA E^0$ given by
    \begin{align*}
        s_E(e^i) &= 
        \begin{cases}
            s_G(e)^1    & \textrm{if}~s_G(e)\neq v, \\
            v^i         & \textrm{if}~s_G(e) = v,
        \end{cases}\\
        r_E(e^i) &= 
        \begin{cases}
            r_G(e)^1   & \textrm{if}~r_G(e) \neq v, \\
            v^j        & \textrm{if}~r_G(e) = v, e\in \E^v_j,
        \end{cases}
    \end{align*}
    for $e^i\in E^1$.
\end{definition}

\begin{remark}
  The above definition is a modification of~\cite[Section 5]{Bates-Pask2004} in that we allow the partition sets to be empty,
  this is called Move \texttt{(I-)} in~\cite[Definition 3.3]{Eilers-Ruiz}.
    Note that this in-split introduces new sources, one for each empty partition set.
    The out-split introduces no new sources.
\end{remark}

\begin{example}
  Consider again the graph of Example~\ref{ex:golden-mean} which encodes the golden mean shift.
  The left-most vertex $v$ has two in-going edges, and we can perform an in-split with respect to the partitions $r^{-1}(v) =\{e\} \amalg \{g\}$ and $r^{-1}(v) = \{e,g\} \amalg \emptyset$.
  This produces the two graphs 
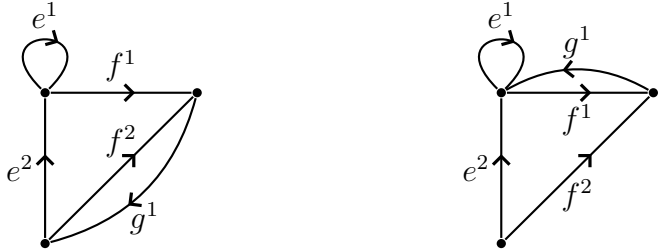
\begin{figure}[H]
    \begin{center}
    \begin{tikzpicture}
    [scale=5, ->-/.style={thick, decoration={markings, mark=at position 0.6 with {\arrow{Straight Barb[line width=0pt 1.5]}}},postaction={decorate}},
        node distance =2cm,
    thick,
    vertex/.style={inner sep=0pt, circle, fill=black}]
        \node(E) {};
        \node[vertex, right of = E] (E1) {.};
        \node[vertex, right of = E1] (E2) {.};
        \node[vertex, below of = E1] (E3) {.};

        \draw[->-, looseness=30, out=135, in=45] (E1) to node[above] {$e^1$} (E1);
        \draw[->-] (E3) to node[left] {$e^2$} (E1);
        \draw[->-, bend left] (E2) to node[below] {$g^1$} (E3);
        \draw[->-] (E1) to node[above] {$f^1$} (E2);
        \draw[->-] (E3) to node[above] {$f^2$} (E2);
        
        \node[right of = E2] (F) {};
        \node[vertex, right of = F] (F1) {.};
        \node[vertex, right of = F1] (F2) {.};
        \node[vertex, below of = F1] (F3) {.};

        \draw[->-, looseness=30, out=135, in=45] (F1) to node[above] {$e^1$} (F1);
        \draw[->-] (F3) to node[left] {$e^2$} (F1);
        \draw[->-, bend right] (F2) to node[above] {$g^1$} (F1);
        \draw[->-] (F1) to node[below] {$f^1$} (F2);
        \draw[->-] (F3) to node[right, below] {$f^2$} (F2);
   \end{tikzpicture}
    \end{center}
\end{figure}
\noindent respectively.
Note that the second in-split which used an empty partition set produced a source in the right-most graph.
\end{example}

The next definition is from~\cite[Definition 3.6]{Eilers-Ruiz}.

\begin{definition}[Move \texttt{(I+)}]
    Let $G$ be a finite graph with no sinks.
    An \emph{elementary balanced in-split} of $G$ is a pair of in-split graphs $E$ and $F$ of $G$ at the same vertex using the same number of partition sets.
\end{definition}

Suppose $E$ and $F$ are elementary balanced in-split graphs of $G$ at $v\in G^0$.
We shall see in Proposition~\ref{prop:I+,eventual-conjugacy} that $E$ and $F$ are eventually conjugate.
The labelings on the vertices and edges (as in Definition~\ref{def:(I-)}) define canonical bijections $\phi\colon E^0 \LRA F^0$ and $\psi^{(0)}\colon E^1 \LRA F^1$ 
given by $\phi(v_E^i) = v_F^i$ for $v_E^i\in E^0$ and $\psi^{(0)}(e^i_E) = e^i_F$ for $e_E^i\in E^1$, respectively.
In general, $\psi^{(0)}$ is not a block map since it is not compatible with $E$ and $F$.
\emph{We shall identify the vertices and edges of $E$ and $F$ via these bijections.}

Given $v^i_E\in E^0$ we can perform an in-split of $E$ at $v^i_E$ using $n$ partition elements to obtain a graph $E_{(2)}$.
Similarily, we obtain a graph $F_{(2)}$ as an in-split of $F$ at $v^i_F\in F^0$ using $n$ partition elements.
Again there is a canonical identification of the vertices and edges in the two graphs.
Proposition~\ref{prop:I+,eventual-conjugacy} below shows that $E_{(2)}$ and $F_{(2)}$ are eventually conjugate
and we say that they are a \emph{2-step balanced in-split of $G$}.
Iterating this process, we define an $\ell$-step balanced in-split recursively.

\begin{definition}[Iterated \texttt{(I+)}]
    Let $G$ be a finite graph with no sinks and let $\ell\in \N$.
    For $\ell\geq 2$, two graphs $E_{(\ell)}$ and $F_{(\ell)}$ are \emph{$\ell$-step balanced in-split graphs of $G$} if the following holds:
    \begin{itemize}
        \item There are graphs $E_{(\ell - 1)}$ and $F_{(\ell - 1)}$ which are $(\ell - 1)$-step balanced in-splits of $G$;
        \item $E_{(\ell)}$ is the in-split graph of $E_{(\ell - 1)}$ at a vertex $v^i_{E_{(\ell - 1)}}\in E_{(\ell - 1)}^0$ using $m$ partition elements; and
        \item $F_{(\ell)}$ is the in-split graph of $F_{(\ell - 1)}$ at the identified vertex $v^i_{F_{(\ell - 1)}}\in F_{(\ell - 1)}^0$ using $m$ partition elements.
    \end{itemize}
    For $\ell = 0$, we require that $E_{(0)} = G = F_{(0)}$ and 
    for $\ell = 1$ the graphs $E_{(1)}$ and $F_{(1)}$ should be elementary balanced in-splits of $G$.
\end{definition}

For $\ell = 2$, we depict this as
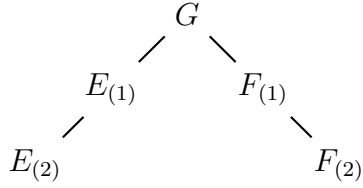
\begin{figure}[H]
    \begin{center}
    \begin{tikzpicture}
    [scale=5, 
    node distance =1cm,
    thick,
    vertex/.style={inner sep=0pt, circle, fill=black}]
        \node (G) {$G$};
        \node[below of = G] (N) {};
        \node[left of = N] (E1) {$E_{(1)}$};
        \node[right of = N] (F1) {$F_{(1)}$};
        \node[below of = E1] (NE1) {};
        \node[below of = F1] (NF1) {};
        \node[left of = NE1] (E2) {$E_{(2)}$};
        \node[right of = NF1] (F2) {$F_{(2)}$};
        
        \draw (G) to (E1);
        \draw (G) to (F1);
        \draw (E1) to (E2);
        \draw (F1) to (F2);
    \end{tikzpicture}
    \end{center}
\end{figure}

\begin{proposition}\label{prop:I+,eventual-conjugacy}
    Let $G$ be a finite graph with no sinks and let $\ell\in \N$.
    If $E_{(\ell)}$ and $F_{(\ell)}$ are $\ell$-step balanced in-split graphs of $G$, then $E_{(\ell)}$ and $F_{(\ell)}$ are $\ell$-conjugate.
    In particular, if $E$ and $F$ are elementary balanced in-split graphs of $G$, then $E$ and $F$ are $1$-conjugate.
\end{proposition}

\begin{proof}
  \emph{Step 1.}
    Suppose $E$ and $F$ are in-split graphs of $G$ at $v\in G^0$ using $n$ sets in the partitions
    \[
        \E^v: r_G^{-1}(v) = \E^v_1 \amalg \cdots \amalg \E^v_n, \qquad
        \F^v: r_G^{-1}(v) = \F^v_1 \amalg \cdots \amalg \F^v_n.
    \]
    There is a canonical surjection $q_{E}\colon E^2 \LRA G^2$ given by
    \[
        q_E(e^i f^j) = e f,
    \]
    for $e^i f^j\in E^2$.
    The map simply forgets the indeces of the edges.
    Similarly, there is a canonical surjection $q_F\colon F^2\LRA G^2$.
    We will construct a map $\psi^{(1)}\colon E^2\LRA F^2$ satisfying $q_E = q_F\circ \psi^{(1)}$.

    If $e f\in G^2$ and $s(e) = v$ then there is a unique bijection $\theta\colon q_E^{-1}(e f) \LRA q_F^{-1}(e f)$ which preserves the label of the first edge, that is,
    \[
        \theta(e^i f^j) = e^i f^{j'}
    \]
    for $e^i f^j\in q_E^{-1}(e f)$.
    If instead $s(e) \neq v$, then the preimage of $e f$ under $q_E$ is a singleton.
    We may therefore define $\psi^{(1)}\colon E^2 \LRA F^2$ as
    \[
        \psi^{(1)}(e^i f^j) = 
        \begin{cases}
            q_F^{-1}(e f) & \textrm{if}~s(e) \neq v, \\
            \theta(e^i f^j) & \textrm{if}~s(e) = v,
        \end{cases}
    \]
    for $e^i f^j \in E^2$.
    This is bijective and compatible and satisfies $q_E = q_F\circ \psi^{(1)}$.
    It therefore induces a $1$-conjugacy between $E^\infty$ and $F^\infty$.

    \emph{Step 2}.
    Set $E_{(1)} = E$ and $F_{(1)} = F$ and suppose $E_{(2)}$ is an in-split graph of $E_{(1)}$ at $w_E\in E_{(1)}^0$ 
    while $F_{(2)}$ is an in-split graph of $F_{(1)}$ at an identified vertex $w_F\in F_{(1)}^0$ using $m$ sets in the partitions
    \[
        \E^w: r_E^{-1}(w) = \E^{w}_1 \amalg \dots \amalg \E^{w}_m, \qquad
        \F^w: r_F^{-1}(w) = \F^{w}_1 \amalg \cdots \amalg \F^{w}_m.
    \]
    We will use the block map $\psi^{(1)}\colon E_{(1)}^2 \LRA F_{(1)}^2$ constructed in \emph{Step 1} to define a $2$-block map $\psi^{(2)}\colon E_{(2)}^3 \LRA F_{(2)}^3$.
    As before there is a canonical surjection $q_{E_{(2)}}\colon E_{(2)}^3 \LRA E_{(1)}^3$ given by
    \[
      q_{E_{(2)}} (e^i f^j g^k) = e f g,
    \]
    for $e^i f^j g^k\in E_{(2)}$.
    The map simply forgets the indeces of the edges.
    Similarily, there is a canonical surjection $q_{F_{(2)}}\colon F_{(2)}^3\LRA F_{(1)}^3$.
    Since $\psi^{(1)}$ is compatible, we can extend it to a $2$-block map $\bar{\psi}^{(1)}\colon E_{(1)}^3 \LRA F_{(1)}^3$.
    We will define a map $\psi^{(2)}\colon E_{(2)}^3\LRA F_{(2)}^3$ such that $\bar{\psi}^{(1)}\circ q_{E_{(2)}} = q_{F_{(2)}}\circ \psi^{(2)}$.

    If $e^i f^j g^k\in E_{(2)}^3$ and $s(e) = w_E$, then there is a unique bijection $\theta\colon q_{E_{(2)}}^{-1}(e f g) \LRA q_{F_{(2)}}^{-1}(\bar{\psi}^{(1)}(e f g))$ 
    which preserves the label of the first index, that is,
    \[
        \theta(e^i f^j g^k) = e^i f^{j'} g^{k'},
    \]
    for $e^i f^j g^k\in E_{(2)}^3$.
    If instead $s(e) \neq w_E$, then $s(\bar{\psi}^{(1)})(e f g) \neq w_F$, and the preimage of $e f g\in E_{(1)}$ under $q_{E_{(2)}}$ is a singleton,
    and the preimage of $\bar{\psi}^{(1)}(e f g)$ under $q_{F_{(2)}}$ is a singleton.
    We may therefore define $\psi^{(2)}\colon E_{(2)}^3\LRA F_{(2)}^3$ by
    \[
        \psi^{(2)}(e^i f^j g^k) =
        \begin{cases}
            q_{F_{(2)}}^{-1}(\bar{\psi}^{(1)}(e f g)) & \textrm{if}~s(e) \neq w_E, \\
            \theta(e^i f^j g^k) & \textrm{if}~s(e) = w_E,
        \end{cases}
    \]
    for $e^i f^j g^k\in E_{(2)}^3$.
    Then $\psi^{(2)}$ is bijective and compatible and satisfies $\bar{\psi}^{(1)}\circ q_{E_{(2)}} = q_{F_{(2)}}\circ \psi^{(2)}$.
    It induces a $2$-conjugacy between $E_{(2)}^\infty$ and $F_{(2)}^\infty$.

    \emph{Step $\ell$}:
    Assume now that $E_{(\ell)}$ and $F_{(\ell)}$ are $\ell$-step balanced in-split graphs of $G$.
    In particular, there is an $(\ell - 1)$-block map $\psi^{(\ell - 1)}\colon E_{(\ell - 1)}^\ell \LRA F_{(\ell - 1)}^\ell$.
    This may be extended to a block map $\bar{\psi}^{(\ell - 1)}\colon E_{(\ell - 1)}^{\ell + 1}\LRA F_{(\ell - 1)}^{\ell + 1}$,
    and there are canonical surjections $q_{E_{(\ell)}}\colon E_{(\ell)}^{\ell + 1} \LRA E_{(\ell - 1)}^{\ell + 1}$ 
    and $q_{F_{(\ell)}}\colon F_{(\ell)}^{\ell + 1}\LRA F_{(\ell - 1)}^{\ell + 1}$ as in Step 2.
    Using the same strategy as above, we can define a block map $\psi^{(\ell)}\colon E_{(\ell)}^{\ell + 1}\LRA F_{(\ell)}^{\ell + 1}$
    which is bijective and compatible and satisfies $\bar{\psi}^{(\ell - 1)}\circ q_{E_{(\ell)}} = q_{F_{(\ell)}}\circ \psi^{(\ell)}$.
    It therefore induces an $\ell$-conjugacy between $E_{(\ell)}^\infty$ and $F_{(\ell)}^\infty$.
\end{proof}

\begin{example}
    It is well-known that one-sided conjugacy is strictly stronger that eventual conjugacy, cf. e.g.~\cite[Example 3.6]{BC}.
    This example shows that the two relations can be distinguished among graphs with finite path spaces.
    Consider the graphs
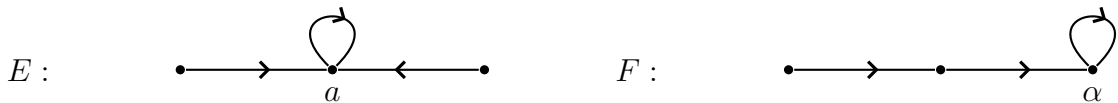
\begin{figure}[H]
    \begin{center}
    \begin{tikzpicture}
    [scale=5, ->-/.style={thick, decoration={markings, mark=at position 0.6 with {\arrow{Straight Barb[line width=0pt 1.5]}}},postaction={decorate}},
    node distance =2cm,
    thick,
    vertex/.style={inner sep=0pt, circle, fill=black}]
        \node (E) {$E:$};
        \node[vertex, right of = E] (E1) {.};
        \node[vertex, right of = E1, label=below:{$a$}] (E2) {.};
        \node[vertex, right of = E2] (E3)  {.};

        \draw[->-] (E1) to (E2);
        \draw[->-] (E3) to (E2);
        \draw[->-, looseness=30, out=135, in=45] (E2) to (E2);
        
        \node (F) [right of = E3] {$F:$};
        \node[vertex, right of = F] (F1) {.};
        \node[vertex, right of = F1] (F2) {.};
        \node[vertex, right of = F2, label=below:{$\alpha$}] (F3)  {.};

        \draw[->-] (F1) to (F2);
        \draw[->-] (F2) to (F3);
        \draw[->-, looseness=30, out=135, in=45] (F3) to (F3);
    \end{tikzpicture}
    \end{center}
\end{figure}
\noindent The adjacency matrices of the graphs have different total column amalgamations, so they are not conjugate.
However, we can construct $E$ and $F$ as a balanced in-split of the graph  
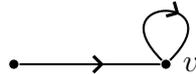
\begin{figure}[H]
    \begin{center}
    \begin{tikzpicture}
    [scale=5, ->-/.style={thick, decoration={markings, mark=at position 0.6 with {\arrow{Straight Barb[line width=0pt 1.5]}}},postaction={decorate}},
    node distance =2cm,
    thick,
    vertex/.style={inner sep=0pt, circle, fill=black}]
        \node[vertex] (G1) {.};
        \node[vertex, right of = G1, label=right:{$v$}] (G2) {.};

        \draw[->-] (G1) to (G2);
        \draw[->-, looseness=30, out=135, in=45] (G2) to (G2);
    \end{tikzpicture}
    \end{center}
\end{figure}
\noindent at the vertex $v$.
Hence $E$ and $F$ are $1$-conjugate by Proposition~\ref{prop:I+,eventual-conjugacy}.
In fact, any bijection of the path spaces which maps the loop in $E$ to the loop in $F$ is an explicit $1$-conjugacy.
\end{example}

Next, we prove a converse to Proposition~\ref{prop:I+,eventual-conjugacy}.
The proof uses ideas from~\cite[Section 2]{Boyle-Franks-Kitchens} (see also~\cite[Chapter 2]{Kitchens}).

\begin{theorem}\label{thm:ec-iterated-balanced}
    Let $E$ and $F$ be finite directed graphs with no sinks.
    Then $E$ and $F$ are eventually conjugate if and only if 
    there exists a graph $G$ such that $E$ and $F$ are (conjugate to) the iterated balanced in-split graphs of $G$.
\end{theorem}

\begin{proof}
    If $E$ and $F$ are the iterated balanced in-split graphs of $G$, then they are eventually conjugate by Proposition~\ref{prop:I+,eventual-conjugacy}.
    For the converse implication suppose that $E$ and $F$ are eventually conjugate.
    We may assume that there is a homeomorphism $h\colon E^\infty \LRA F^\infty$ which is an $\ell$-conjugacy and 
    whose inverse $h^{-1}\colon F^\infty \LRA E^\infty$ is also an $\ell$-conjugacy
    and that $c\in \N$ is a continuity constant for $h$ and $h^{-1}$.

    We start by constructing the graph $G$ as follows:
    The set of vertices is given by 
    \[
        G^0 = \{ \big(x_{[\ell, \ell + 2c]}, y_{[\ell, \ell + 2c]}\big) \mid x\in E^\infty, y\in F^\infty,  h(x) = y\}
    \]
    with the following transition rule:
    There is an edge from $(x_{[\ell, \ell + c]}, y_{[\ell, \ell + c]})$ to $(z_{[\ell, \ell + c]}, w_{[\ell, \ell + c]})$ 
    if and only if $x_{(\ell, \ell + c]} = z_{[\ell, \ell + c)}$ and $y_{(\ell, \ell + c]} = w_{[\ell, \ell + c)}$.

    For each $j = 0,\ldots,\ell$, define the graph $E_{(j)}$ with the following set of vertices
    \[
        E_{(j)}^0 = \{ \big(x_{[\ell - j, \ell + 2c]}, y_{[\ell, \ell + 2c]}\big) \mid x\in E^\infty, y\in F^\infty,  h(x) = y \}
    \]
    and a transition rule similar to the above.
    Define also the graph $F_{(j)}$ with vertices
    \[
        F_{(j)}^0 = \{ \big(x_{[\ell, \ell + 2c]}, y_{[\ell - j, \ell + 2c]}\big) \mid x\in E^\infty, y\in F^\infty,  h(x) = y \}
    \]
    and a similar transition rule.
    Then $E_{(0)} = G = F_{(0)}$.
    We will show that $E_{(\ell)}$ and $F_{(\ell)}$ are iterated balanced in-splits of $G$.

    Whenever $\big( x_{[\ell - j, \ell + 2c]}, y_{[\ell, \ell + 2c]} \big)$ is a vertex in $E_{(j)}$ and
    \[
        \big( a x_{[\ell - j, \ell + 2c]}, y_{[\ell, \ell + 2c]}\big), 
        \big( a' x_{[\ell - j, \ell + 2c]}, y_{[\ell, \ell + 2c]}\big),
    \]
    are distinct vertices in $E_{(j + 1)}$, then the vertices have the same future and distinct pasts.
    It follows that $E_{(j + 1)}$ in an in-split graph of $E_{(j)}$ for $j = 0, \ldots, \ell - 1$.
    Fix a vertex $\big(x_{[\ell, \ell + 2c]}, y_{[\ell, \ell + 2c]}\big)\in G^0$ and consider the set
    \[
        A = \{ z\in E^\infty \mid z_{[\ell, \ell + 2c]} = x_{[\ell, \ell + 2c]}, h(z)_{[\ell, \ell + 2c]} = y_{[\ell, \ell + 2c]}\}.
    \]
    It follows from the choice of $c\in \N$ that if $z, z'\in A$ and $h(z)_{[0, \ell)} = h(z')_{[0, \ell)}$ then $z_{[0, \ell)} = z'_{[0, \ell)}$.
    By a symmetric argument using $h^{-1}$, there is a bijection between the words of length $\ell$ which can be appended to $x_{[\ell, \ell + 2c]}$ 
    and the words of length $\ell$ which can be appended to $y_{[\ell, \ell + 2c]}$.
    In particular, if $\ell = 1$ then $E_{(1)}$ and $F_{(1)}$ are balanced in-splits of $G$.

    Assume now that $\ell > 0$ and take $z, z'\in A$ with $z\in Z_E(a_{[0, \ell)})$ and $z'\in Z_E(a_{[0,\ell)})$,
    for some words $a_{[0,\ell)}, a'_{[0,\ell)}\in E^\ell$ with $a_{[1,\ell)} \neq a'_{[1,\ell)}$.
    By the choice of $c$, we see that $h(\sigma_E(z))_{[0, \ell)} \neq h(\sigma_E(z'))_{[0, \ell)}$.
    It follows that there is a bijection between the set of words of length $\ell - 1$ which can be appended to $x_{[\ell, \ell + 2c]}$
    and the set of words of length $\ell - 1$ which can be appended to $y_{[\ell, \ell + 2c]}$.
    Continuing this process we see that for each $j = 0,\ldots, \ell$, the graphs $E_{(\ell - j)}$ and $F_{(\ell - j)}$ are iterated balanced in-splits of $G$.
    In particular, $E_{(\ell)}$ and $F_{(\ell)}$ are $\ell$-step balanced in-splits of $G$.

    It remains to verify that $E_{(\ell)}$ can be reached from $E$ by a finite sequence of out-splits,
    and that $F_{(\ell)}$ can be reached from $F$ by a finite sequence of out-splits.

    For each $i = 0, \ldots, 2c$, consider the graph $E^{(\ell + i)}$ with the set of vertices
    \[
        {(E^{(\ell + i)})}^0 = \big\{ \big(x_{[0, \ell + 2c]}, y_{[\ell, \ell + i]}\big) \mid h(x) = y\big\}
    \]
    and an overlapping transition rule similar to the one described above.
    Then $E^{(\ell)}$ and $E^{[\ell + 2c + 1]}$ are graph isomorphic,
    and it is well-known that $E^{[\ell + 2c + 1]}$ is conjugate to $E$ via a sequence of out-splits, see, e.g.,~\cite[Section 2.4]{LM}.
    Furthermore, $E^{(\ell + (i + 1))}$ is constructed from $E^{(\ell + i)}$ by a sequence of out-splits.
    Indeed, if $\big( x_{(0, \ell + 2c]} z, w_{[\ell, \ell + i]} \big)$ and $\big( x_{(0, \ell + 2c]} z', w'_{[\ell, \ell + i]} \big)$ are distinct vertices 
    which follow $\big( x_{[0, \ell + 2c]}, y_{[\ell, \ell + i]} \big)$ in $E^{(\ell)}$,
    then the latter vertex splits into distinct vertices $\big( x_{[0, \ell + 2c]} z, y_{[\ell, \ell + i]} w_{\ell + i} \big)$ and 
    $\big( x_{[0, \ell + 2c]}, y_{[\ell, \ell + i]} w'_{\ell + i} \big)$ in $E^{(\ell + 1)}$ with identical pasts but distinct futures.
    A similar argument shows that $F_{(\ell)}$ can be reached from $F$ via a sequence of out-splits.
    
    Hence $E$ and $F$ are conjugate to graphs which are the iterated balanced in-split of the graph $G$.
\end{proof}

For $\ell = 0$, we have an immediate consequence.

\begin{corollary}
    One-sided conjugacy among finite graphs with no sinks is generated by out-splits (Move (\texttt{O})).
\end{corollary}

Finally, we will show that $\ell$-step balanced in-split graphs can be connected by a sequence of elementary balanced in-splits.

\begin{theorem}\label{thm:generated}
    Eventual conjugacy of finite graphs with no sinks is generated by out-splits (Move (\texttt{O})) and elementary balanced in-splits (Move ($\texttt{I+}$)).
\end{theorem}

\begin{proof}
    Suppose $E$ and $F$ are eventually conjugate graphs.
    We know from Theorem~\ref{thm:ec-iterated-balanced} that $E$ and $F$ are conjugate (via a sequence of out-splits) to graphs which
    are $\ell$-step balanced in-splits of a graph $G$.
    We may assume that $\ell\geq 2$.
    If $E$ and $F$ are $\ell$-step balanced in-splits of $G$,
    we will show that they can be connected by a finite sequence of elementary balanced in-splits.

    Let $(E_{(1)}, \ldots,E_{(\ell)})$ and $(F_{(1)}, \ldots,F_{(\ell)})$ be an $\ell$-step balanced in-split connecting $E_{(\ell)} = E$ and $F_{(\ell)} = F$.
    More specifically, for $i = 0,\ldots,\ell - 1$, the graph $E_{(i + 1)}$ is an in-split of $E_{(i)}$ at the vertex $v_{(i)}$ using $n_{(i)}$ sets.
    By a slight abuse of notation, we use the same symbol $v_{(i)}$ for the vertex in $F_{(i)}$ which is being in-split using $n_{(i)}$ sets to construct $F_{(i + 1)}$.
    We let $q_{E_{(i + 1)}}\colon E_{(i + 1)}^0 \LRA E_{(i)}^0$ denote the canonical surjection of vertices which simply forgets the labeling.

    Construct a graph $E_{(\ell - 1, \ell - 2)}$ as an in-split of $E_{(\ell - 1)}$ at $v_{(\ell - 1)}$ using $n_{(\ell - 1)}$ sets
    where every edge is placed in the first partition set.
    This introduces $n_{(\ell - 1)} - 1$ sources
    \[
        \{v_{(\ell - 1)}^2,\ldots,v_{(\ell - 1)}^{n_{(\ell - 1)}}\} \subset E_{(\ell - 1, \ell - 2)}^0,
    \]
    and the graphs $E_{(\ell - 1, \ell - 2)}$ and $E_{(\ell)}$ are elementary balanced in-splits of $E_{(\ell - 1)}$, by construction.
    Next we construct a graph $E_{(\ell - 2)}'$ as $E_{(\ell - 2)}$ with sources attached.
    More precisely, 
    \begin{align*}
        {(E_{(\ell - 2)}')}^0 &= E_{(\ell - 2)}^0 \cup \{v_{(\ell - 1)}^j \mid j = 2, \ldots, n_{(\ell - 1)} \},  \\
        {(E_{(\ell - 2)}')}^1 &= E_{(\ell - 2)}^1 \cup \{ e^j \mid  j = 2,\ldots,n_{(\ell - 2)}, e\in s_{E_{(\ell - 1)}}^{-1}(v_{(\ell - 1)}) \},
    \end{align*}
    with $s_{E_{(\ell - 2)}'}(e^j) = v_{(\ell - 1)}^j$ and $r_{E_{(\ell - 2)}'}(e^j) = q_{E_{(\ell - 1)}}(r_{E_{(\ell - 1)}}(e))$.
    The notation reflects the fact that $E_{(\ell - 1, \ell - 2)}$ is an in-split of $E_{(\ell - 1)}$ (at $v_{(\ell - 1)}$) 
    and an in-split of $E_{(\ell - 2)}'$ (at $v_{(\ell - 2)}$).

    If $\ell = 2$, we apply a similar procedure to $F_{(\ell)}$ to obtain a graph $F_{(\ell - 2, \ell - 1)} = F_{(0,1)}$,
    and $F_{(0,1)}$ and $F_{(2)}$ are elementary balanced in-splits of $F_{(1)}$.
    Furthermore, $E_{(1,0)}$ and $F_{(0,1)}$ are elementary balanced in-splits of $G' = E_{(0)}' = F_{(0)}'$ at $v_{(0)}$ using $n_{(0)}$ sets.

    If $\ell > 2$, construct the graph $E_{(\ell - 2, \ell - 3)}$ as the in-split of $E_{(\ell - 2)}'$ at $v_{(\ell - 2)}$ using $n_{(\ell - 2)}$ sets
    where every edge is placed in the first partition set.
    This introduces $n_{(\ell - 2)} - 1$ sources, and $E_{(\ell - 1, \ell - 2)}$ and $E_{(\ell - 2, \ell - 3)}$ are elementary balanced in-splits of $E_{(\ell - 2)}'$.
    Next construct the graph $E_{(\ell - 3)}'$ as $E_{(\ell - 3)}$ with additional sources attached as follows:
    \begin{align*}
        {(E_{\ell - 3}')}^0 &= 
        E_{\ell - 3}^0 \cup \{v_{\ell - 1}^j \mid j = 2,\ldots, n_{(\ell - 1)} \} \cup \{v_{(\ell - 2)}^k \mid k = 2,\ldots, n_{(\ell - 2)} \}, \\
        {E_{(\ell - 3)'}}^1 &=
            E_{(\ell - 3)}^1 \cup \{e^j \mid j = 2,\ldots, n_{(\ell - 1)}, e\in s_{E_{(\ell - 1)}}^{-1}(v_{(\ell - 1)}) \} \\
            & \quad \cup \{ f^k \mid k = 2,\ldots,n_{(\ell - 2)}, f\in s_{E_{(\ell - 2)}}^{-1}(v_{(\ell - 2)})\}
    \end{align*}
    with $s(e^j) = v_{(\ell - 1)}^j$ and $s(f^k) = v_{(\ell - 2)}^k$, and
    \[
        r(e^j) = q_{E_{(\ell - 2)}}\circ q_{E_{(\ell - 1)}}(r_{E_{(\ell - 1)}}(e)), \quad
        r(f^k) = q_{E_{(\ell - 2)}}(r_{E_{(\ell - 2)}}(f)).
    \]
    Note that $E_{(\ell - 2, \ell - 3)}$ is an in-split of $E_{(\ell - 3)}'$ (at $v_{(\ell - 3)}$).
    
    Applying the same procedure starting in $F_{(\ell)}$ and iterating the process, we finally obtain a graph $G' = E_{(0)}' = F_{(0)}'$ with $\sum_{i = 0}^{\ell - 1} (n_{(i)} - 1)$ sources attached,
    and graphs $E_{(1,0)}$ and $F_{(0,1)}$ which are elementary balanced in-splits of $G'$.
    Therefore, $E_{(\ell)}$ and $F_{(\ell)}$ are connected by $2 \ell - 1$ elementary balanced in-splits.
\end{proof}

\begin{remark}\label{rem:picture}
For $\ell = 2$, the procedure in the proof above looks like this
\begin{figure}[H]
    \begin{center}
    \begin{tikzpicture}
    [scale=5, 
    node distance =1cm,
    thick,
    vertex/.style={inner sep=0pt, circle, fill=black}]
        \node (G) {$G$};
        \node[below of = G] (G1) {$G'$};
        \node[left of = G1] (N1) {};
        \node[right of = G1] (N2) {};
        \node[right of = N2] (F1) {$F_{(1)}$};
        \node[left of = N1] (E1) {$E_{(1)}$};
        \node[below of = E1] (NE1) {};
        \node[below of = F1] (NF1) {};
        \node[left of = NE1] (E2) {$E_{(2)}$};
        \node[right of = NE1] (E2') {$E_{(1,0)}$};
        \node[right of = NF1] (F2) {$F_{(2)}$};
        \node[left of = NF1] (F2') {$F_{(0, 1)}$};
        
        \draw (G) to (E1);
        \draw (G) to (F1);
        \draw (E1) to (E2);
        \draw (F1) to (F2);
        \draw (E1) to (E2');
        \draw (F1) to (F2');
        \draw (G1) to (E2');
        \draw (G1) to (F2');
    \end{tikzpicture}
    \end{center}
\end{figure}
\noindent We illustrate this in Example~\ref{ex:example2} below.
\end{remark}

\begin{remark}
    In the proof of Theorem~\ref{thm:generated} it is crucial to make use of graphs with sources in order to connect $\ell$-step balanced in-split graphs by elementary balanced in-splits.
    This indicated the usefulness of considering the class of finite graphs with no sources instead or the smaller class of essential finite graphs.
    We do not know if it possible to prove this result circumventing graphs with sources.
\end{remark}

\begin{example}\label{ex:example2}
    In this example we generate $\ell$-conjugate graphs.
    Consider the graph
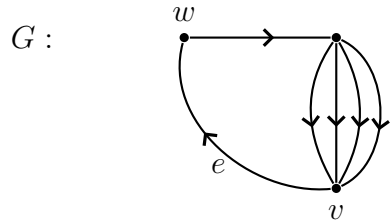
\begin{figure}[H]
    \begin{center}
    \begin{tikzpicture}
    [scale=5, ->-/.style={thick, decoration={markings, mark=at position 0.6 with {\arrow{Straight Barb[line width=0pt 1.5]}}},postaction={decorate}},
    node distance =2cm,
    thick,
    vertex/.style={inner sep=0pt, circle, fill=black}]
        \node (G) {$G:$};
        \node[vertex, right of = G, label=above:{$w$}] (G1) {.};
        \node[vertex, right of = G1] (G2) {.};
        \node[vertex, below of = G2, label=below:{$v$}] (G3) {.};

        \draw[->-] (G1) to (G2);
        \draw[->-] (G2) to (G3);
        \draw[->-, bend right, out=325] (G2) to (G3);
        \draw[->-, bend left] (G2) to (G3);
        \draw[->-, bend left, out=65, in=115] (G2) to (G3);
        \draw[->-, bend left, out=50, in=120] (G3) to node[below] {$e$} (G1);
    \end{tikzpicture}
    \end{center}
\end{figure}
\noindent where $e$ is a path of length $\ell - 1$.
If we perform balanced in-splits along the vertices of the path $e$ ending at $w$ we will obtain $\ell$-conjugate graphs.
The case with $\ell = 1$ and $v = w$ appeared in~\cite[Example 3.6]{BC}.

We illustrate the procedure when $\ell = 2$ and $e$ is a path of length one.
A balanced in-split at $v$ is given by
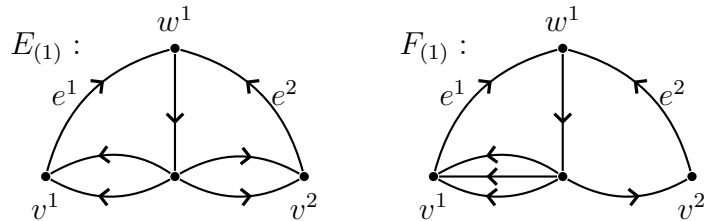
\begin{figure}[H]
    \begin{center}
    \begin{tikzpicture}
    [scale=5, ->-/.style={thick, decoration={markings, mark=at position 0.6 with {\arrow{Straight Barb[line width=0pt 1.5]}}},postaction={decorate}},
    node distance =1.7cm,
    thick,
    vertex/.style={inner sep=0pt, circle, fill=black}]
    \node (E) {$E_{(1)}:$};
        \node[vertex, right of = E, label=above:{$w^1$}] (E1) {.};
        \node[vertex, below of = E1] (E2) {.};
        \node[vertex, left of = E2, label=below:{$v^1$}] (E3) {.};
        \node[vertex, right of = E2, label=below:{$v^2$}] (E4) {.};
        \node[right of = E1] (E5) {};

        \draw[->-] (E1) to (E2);
        \draw[->-, bend left] (E2) to (E3);
        \draw[->-, bend right, out=325] (E2) to (E3);
        \draw[->-, bend right] (E2) to (E4);
        \draw[->-, bend left] (E2) to (E4);
        \draw[->-, bend left] (E3) to node[left] {$e^1$} (E1);
        \draw[->-, bend right] (E4) to node[right] {$e^2$} (E1);

        \node (F) [right of = E5] {$F_{(1)}:$};
        \node[vertex, right of = F, label=above:{$w^1$}] (F1) {.};
        \node[vertex, below of = F1] (F2) {.};
        \node[vertex, left of = F2, label=below:{$v^1$}] (F3) {.};
        \node[vertex, right of = F2, label=below:{$v^2$}] (F4) {.};

        \draw[->-] (F1) to (F2);
        \draw[->-, bend left] (F2) to (F3);
        \draw[->-, bend right, out=325] (F2) to (F3);
        \draw[->-, bend right] (F2) to (F4);
        \draw[->-] (F2) to (F3);
        \draw[->-, bend left] (F3) to node[left] {$e^1$} (F1);
        \draw[->-, bend right] (F4) to node[right] {$e^2$} (F1);
    \end{tikzpicture}
    \end{center}
\end{figure}
Next, we perform a balanced in-split at $w^1$ to get

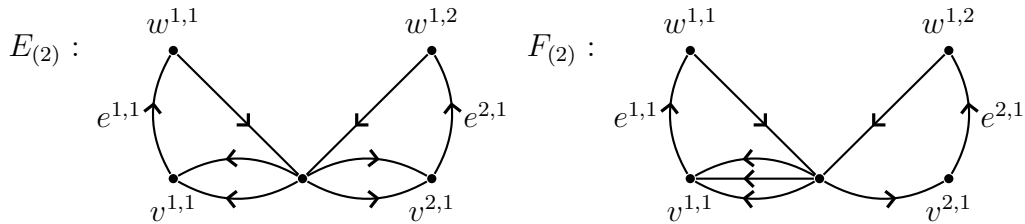
\begin{figure}[H]
    \begin{center}
    \begin{tikzpicture}
    [scale=5, ->-/.style={thick, decoration={markings, mark=at position 0.6 with {\arrow{Straight Barb[line width=0pt 1.5]}}},postaction={decorate}},
    node distance =1.7cm,
    thick,
    vertex/.style={inner sep=0pt, circle, fill=black}]
    \node (E) {$E_{(2)}:$};
        \node[vertex, right of = E, label=above:{$w^{1,1}$}] (E1) {.};
        \node[right of = E1] (N1) {};
        \node[vertex, right of = N1, label=above:{$w^{1, 2}$}] (E2) {.};
        \node[vertex, below of = N1] (E3) {.};
        \node[below of = N1] (N3) {};
        \node[vertex, left of = N3, label=below:{$v^{1,1}$}] (E4) {.};
        \node[vertex, right of = N3, label=below:{$v^{2,1}$}] (E5) {.};
        \node[right of = E2] (N2) {};

        \draw[->-] (E1) to (E3);
        \draw[->-] (E2) to (E3);
        \draw[->-, bend left] (E3) to (E4);
        \draw[->-, bend right] (E3) to (E4);
        \draw[->-, bend right] (E3) to (E5);
        \draw[->-, bend left] (E3) to (E5);
        \draw[->-, bend left] (E4) to node[left] {$e^{1,1}$} (E1);
        \draw[->-, bend right] (E5) to node[right] {$e^{2,1}$} (E2);

    \node (F) [right of = E2] {$F_{(2)}:$};
    \node[vertex, right of = F, label=above:{$w^{1,1}$}] (F1) {.};
        \node[right of = F1] (M1) {};
        \node[vertex, right of = M1, label=above:{$w^{1,2}$}] (F2) {.};
        \node[vertex, below of = M1] (F3) {.};
        \node[below of = M1] (M3) {};
        \node[vertex, left of = M3, label=below:{$v^{1,1}$}] (F4) {.};
        \node[vertex, right of = M3, label=below:{$v^{2,1}$}] (F5) {.};
        \node[right of = F2] (M2) {};

        \draw[->-] (F1) to (F3);
        \draw[->-] (F2) to (F3);
        \draw[->-, bend left] (F3) to (F4);
        \draw[->-, bend right] (F3) to (F4);
        \draw[->-] (F3) to (F4);
        \draw[->-, bend right] (F3) to (F5);
        \draw[->-, bend left] (F4) to node[left] {$e^{1,1}$} (F1);
        \draw[->-, bend right] (F5) to node[right] {$e^{2,1}$} (F2);
    \end{tikzpicture}
    \end{center}
\end{figure}
\noindent and the graphs $E_{(2)}$ and $F_{(2)}$ are $2$-conjugate.

Finally, we use the procedure of the proof of Theorem~\ref{thm:generated} for this example, see the picture of Remark~\ref{rem:picture}.
Observe that the graphs $E_{(1,0)}$ and $F_{(0,1)}$ below
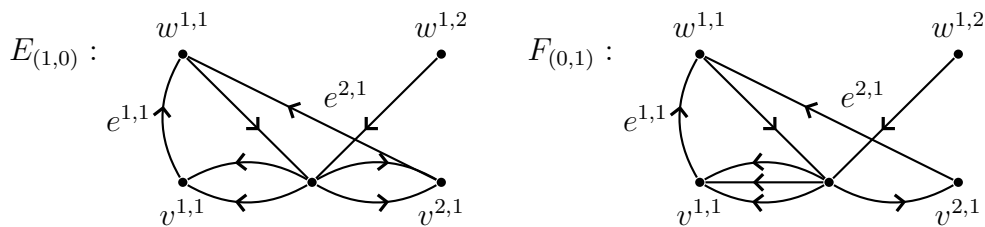
\begin{figure}[H]
    \begin{center}
    \begin{tikzpicture}
    [scale=5, ->-/.style={thick, decoration={markings, mark=at position 0.6 with {\arrow{Straight Barb[line width=0pt 1.5]}}},postaction={decorate}},
    node distance =1.7cm,
    thick,
    vertex/.style={inner sep=0pt, circle, fill=black}]
    \node (E) {$E_{(1,0)}:$};
        \node[vertex, right of = E, label=above:{$w^{1,1}$}] (E1) {.};
        \node[right of = E1] (N1) {};
        \node[vertex, right of = N1, label=above:{$w^{1, 2}$}] (E2) {.};
        \node[vertex, below of = N1] (E3) {.};
        \node[below of = N1] (N3) {};
        \node[vertex, left of = N3, label=below:{$v^{1,1}$}] (E4) {.};
        \node[vertex, right of = N3, label=below:{$v^{2,1}$}] (E5) {.};
        \node[right of = E2] (N2) {};

        \draw[->-] (E1) to (E3);
        \draw[->-] (E2) to (E3);
        \draw[->-, bend left] (E3) to (E4);
        \draw[->-, bend right] (E3) to (E4);
        \draw[->-, bend right] (E3) to (E5);
        \draw[->-, bend left] (E3) to (E5);
        \draw[->-, bend left] (E4) to node[left] {$e^{1,1}$} (E1);
        \draw[->-] (E5) to node[above right] {$e^{2,1}$} (E1);

    \node (F) [right of = E2] {$F_{(0,1)}:$};
    \node[vertex, right of = F, label=above:{$w^{1,1}$}] (F1) {.};
        \node[right of = F1] (M1) {};
        \node[vertex, right of = M1, label=above:{$w^{1,2}$}] (F2) {.};
        \node[vertex, below of = M1] (F3) {.};
        \node[below of = M1] (M3) {};
        \node[vertex, left of = M3, label=below:{$v^{1,1}$}] (F4) {.};
        \node[vertex, right of = M3, label=below:{$v^{2,1}$}] (F5) {.};
        \node[right of = F2] (M2) {};

        \draw[->-] (F1) to (F3);
        \draw[->-] (F2) to (F3);
        \draw[->-, bend left] (F3) to (F4);
        \draw[->-, bend right] (F3) to (F4);
        \draw[->-] (F3) to (F4);
        \draw[->-, bend right] (F3) to (F5);
        \draw[->-, bend left] (F4) to node[left] {$e^{1,1}$} (F1);
        \draw[->-] (F5) to node[above right] {$e^{2,1}$} (F1);
    \end{tikzpicture}
    \end{center}
\end{figure}
\noindent are elementary balanced in-splits of
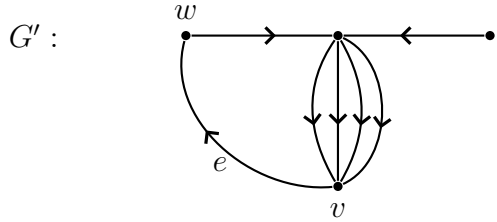
\begin{figure}[H]
    \begin{center}
    \begin{tikzpicture}
    [scale=5, ->-/.style={thick, decoration={markings, mark=at position 0.6 with {\arrow{Straight Barb[line width=0pt 1.5]}}},postaction={decorate}},
    node distance =2cm,
    thick,
    vertex/.style={inner sep=0pt, circle, fill=black}]
        \node (G) {$G':$};
        \node[vertex, right of = G, label=above:{$w$}] (G1) {.};
        \node[vertex, right of = G1] (G2) {.};
        \node[vertex, below of = G2, label=below:{$v$}] (G3) {.};
        \node[vertex, right of = G2] (G4) {.};

        \draw[->-] (G1) to (G2);
        \draw[->-] (G2) to (G3);
        \draw[->-, bend right, out=325] (G2) to (G3);
        \draw[->-, bend left] (G2) to (G3);
        \draw[->-, bend left, out=65, in=115] (G2) to (G3);
        \draw[->-, bend left, out=50, in=120] (G3) to node[below] {$e$} (G1);
        \draw[->-] (G4) to (G2);
    \end{tikzpicture}
    \end{center}
\end{figure}
\noindent Hence $E_{(2)}$ and $F_{(2)}$ are connected by three elementary balanced in-splits.

\end{example}


\section{Balanced strong shift equivalence}
Williams introduced the notion of strong shift equivalence of nonnegative integer matrices and showed that a pair of shifts of finite type are two-sided conjugate 
if and only if their defining matrices are strong shift equivalent.
In this section, we introduce balanced strong shift equivalence of nonnegative integer matrices and show that a pair of finite graphs with no sinks are eventually conjugate
if and only if they are conjugate to graphs whose adjacency matrices are balanced strong shift equivalent.

\begin{definition}
    A \emph{division matrix} is a finite rectangular $\{0,1\}$-matrix in which every column contains exactly one nonzero entry
    and every row contains at least one nonzero entry.
    The transpose of a division matrix is called an \emph{amalgamation matrix}.
\end{definition}

Given a graph $E$, its adjacency matrix is a nonnegative $E^0\times E^0$-matrix $\textbf{A}_E$ whose $(i,j)$-entry is the number of edges in $E$ from $i$ to $j$.
The adjacency matrix of a graph is unique up to a permutation of vertices.
A classical result relates an out-splitting of a graph to their adjacency matrices.
The theorem below is stated in~\cite[Theorem 2.4.14]{LM} for essential graphs but the proof carries through with the existence of sources.

\begin{theorem}\label{thm:out-in-split}
    Let $E$ and $F$ be finite graphs with no sinks.
    Then $F$ is an out-split of $E$ if and only if there exist a division matrix $\mathbf{D}$ and a nonnegative rectangular matrix $\mathbf{E}$ such that
    \[
        \mathbf{A}_E = \mathbf{DE}, \qquad \mathbf{A}_F = \mathbf{ED}.
    \]
    Similarly, $F$ is an in-split of $E$ if and only if there exist a division matrix $\mathbf{D}$ and a nonnegative rectangular matrix $\mathbf{E}$ such that
    \[
        \mathbf{A}_E = \mathbf{E} \mathbf{D}^t, \qquad \mathbf{A}_F = \mathbf{D}^t \mathbf{E}.
    \]
    Here, $\mathbf{A}_E$ and $\mathbf{A}_F$ are the adjacency matrices of $E$ and $F$, respectively.
\end{theorem}

Let $G$ be a finite graph with no sinks and suppose $E$ is an in-split graph of $G$.
Fix an ordering of $E^0$ and $F^0$, and let $|E^0|$ and $|F^0|$ denote the number of vertices in the graphs, respectively.
We can construct a division matrix \textbf{D} as the $|E^0|\times |F^0|$ $\{0,1\}$-matrix subject to the condition that $\textbf{D}(v,w) = 1$ 
if and only if $w\in E^0$ is one of the split vertices of $v\in G^0$.
The matrix \textbf{E} is the $|E^0|\times |F^0|$ nonnegative integer matrix (with the same ordering of the vertices in $E$ and $F$ as for \textbf{D}) for which
$\textbf{E}(v,w)$ is the number of edges in $G$ which are emitted from $i\in G^0$ and contained in the partition set corresponding to $w\in E^0$.
In~\cite{LM}, the matrix \textbf{E} is called the \emph{edge matrix} of the split.

The fact that we allow empty partition sets for in-splits, Move (\texttt{I-}), is reflected in \textbf{E} as zero columns, and it has no effect on \textbf{D}.
From this description we also see that if $F$ is another in-split graph of $G$, then $E$ and $F$ are balanced in-splits if and only if the division matrices agree (up to permutation of vertices).

\begin{proposition}\label{prop:balanced-elementary-eq}
    Let $E$ and $F$ be finite graphs with no sinks.
    Then $E$ and $F$ are elementary balanced in-splits of a graph $G$ if and only if there exist an amalgamation matrix $\mathbf{S}$ and nonnegative rectangular matrices $\mathbf{R}_E$ and $\mathbf{R}_F$
    such that
    \[
        \mathbf{A}_E = \mathbf{S}\mathbf{R}_E, \quad
        \mathbf{A}_F = \mathbf{S}\mathbf{R}_F, \quad
        \mathbf{R}_E \mathbf{S} = \mathbf{R}_F \mathbf{S},
    \]
    where $\textbf{A}_E$ and $\textbf{A}_F$ are the adjacency matrices of $E$ and $F$, respectively.
\end{proposition}

\begin{proof}
    The graph $E$ is an in-split of $G$ if and only if there exist a division matrix $\textbf{D}_E$ and a nonnegative rectangular matrix $\textbf{R}_E$ such that
    $\textbf{A}_E = \textbf{D}_E^t \textbf{R}_E$ and $\textbf{A}_G = \textbf{R}_E \textbf{D}_E^t$.
    Similarly, $F$ is an in-split of $G$ if and only if there exist a division matrix $\textbf{D}_F$ and a nonnegative rectangular matrix $\textbf{R}_F$ such that
    $\textbf{A}_F = \textbf{D}_F^t \textbf{R}_F$ and $\textbf{A}_G = \textbf{R}_F \textbf{D}_F^t$.
    By the observation above, the in-splittings are balanced exactly when $\textbf{D}_E = \textbf{D}_F$.
    Putting $\textbf{S} := \textbf{D}_E^t = \textbf{D}_F^t$ finishes the proof.
\end{proof}

The proposition above is a balanced analog of Theorem~\ref{thm:out-in-split} and this motivates the next definition.

\begin{definition}
    A pair of nonnegative integer matrices \textbf{A} and \textbf{B} are \emph{balanced elementary equivalent} if there exist 
    nonnegative rectangular matrices \textbf{S}, $\textbf{R}_\textbf{A}$ and $\textbf{R}_\textbf{B}$ such that
    \begin{enumerate}
        \item[(i)] $\textbf{A} = \textbf{S} \textbf{R}_\textbf{A}$,
        \item[(ii)] $\textbf{B} = \textbf{S} \textbf{R}_\textbf{B}$, 
        \item[(iii)] $\textbf{R}_\textbf{A} \textbf{S} = \textbf{R}_\textbf{B} \textbf{S}$,
    \end{enumerate}
    in which case the triple $(\textbf{R}_\textbf{A},\textbf{S}, \textbf{R}_\textbf{B})$ is a balanced elementary equivalence from \textbf{A} to \textbf{B}.
    The matrices \textbf{A} and \textbf{B} are \emph{balanced strong shift equivalent} if there is a finite sequence of balanced elementary equivalences connecting \textbf{A} to \textbf{B}.
\end{definition}

Let us discuss some immediate consequences of this definition.
\begin{itemize}
    \item If \textbf{A} and \textbf{B} are balanced elementary equivalent via $(\textbf{R}_\textbf{A}, \textbf{S}, \textbf{R}_\textbf{B})$,
        then \textbf{A} and \textbf{B} are necessarily square with the same dimensions and this allows us to identify the vertices of the graphs determined by \textbf{A} and \textbf{B}.
        The matrices $\textbf{R}_\textbf{A}$ and $\textbf{R}_\textbf{B}$ also have the same dimensions.
        For this reason, balanced strong shift equivalence is too rigid to capture one-sided conjugacy (out-splits) of the graphs.
    \item A zero row in \textbf{S} produces zero rows in \textbf{A} and \textbf{B}. 
        This corresponds to sinks in the corresponding graphs so we do not allow this in the context of this paper.
        However, it might be relevant for a similar description encompassing a broader class of graphs.
    \item In spite of its name, elementary equivalence is not an equivalence relation because it fails to be transitive, and strong shift equivalence is exactly its transitive closure.
        Similarly, balanced elementary equivalence is reflexive and symmetric but not transitive, see Example~\ref{ex:transitive},
        and we have defined balanced strong shift equivalence to be its transitivification.
    \item If \textbf{A} and \textbf{B} are balanced elementary equivalent matrices, then $\det(\textbf{A}) = \det(\textbf{B})$.
        Moreover, $\textbf{A}^{n+1} = \textbf{B}^n \textbf{A}$ and $\textbf{B}^{n+1} = \textbf{A}^n \textbf{B}$, for all $n\in \N_+$.
    \item Given a pair of graphs $E$ and $F$ with adjacency matrices \textbf{A} and \textbf{B}, respectively, we can ask whether they are balanced elementary equivalent
        via some graph with adjacency matrix \textbf{C} whose dimensions must be bounded by the dimensions of \textbf{A} and \textbf{B}.
        This puts a natural bound on the dimensions and entries of a balanced elementary equivalence $(\textbf{R}_\textbf{A}, \textbf{S}, \textbf{R}_\textbf{B})$.
        There are only finitely many choices of such matrices so this relation is decidable. 
        It would be interesting to know whether balanced strong shift equivalence is a decidable relation.
\end{itemize}

\begin{example}\label{ex:transitive}
    This example shows that the elementary balanced in-split relation is not transitive.
    Consider the graphs 
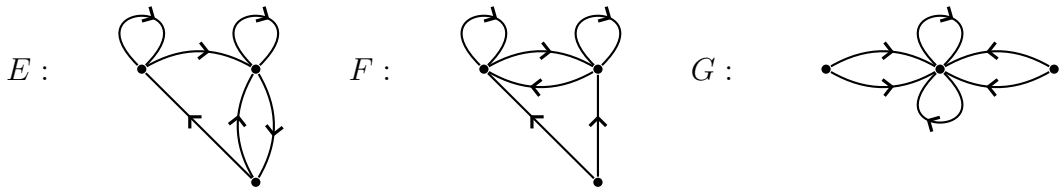
\begin{figure}[H]
    \begin{center}
    \begin{tikzpicture}
    [scale=5, ->-/.style={thick, decoration={markings, mark=at position 0.6 with {\arrow{Straight Barb[line width=0pt 1.5]}}},postaction={decorate}},
        node distance =1.5cm,
    thick,
    vertex/.style={inner sep=0pt, circle, fill=black}]
        \node (E) {$E:$};
        \node[vertex, right of = E] (E1) {.};
        \node[vertex, right of = E1] (E2) {.};
        \node[vertex, below of = E2] (E3) {.};

        \draw[->-, looseness=30, out=135, in=45] (E1) to (E1);
        \draw[->-, looseness=30, out=135, in=45] (E2) to (E2);
        \draw[->-, bend left] (E1) to (E2);
        \draw[->-, bend left] (E2) to (E3);
        \draw[->-] (E3) to (E1);
        \draw[->-, bend left] (E3) to (E2);
        \node[right of = E2] (F) {$F:$};
        \node[vertex, right of = F] (F1) {.};
        \node[vertex, right of = F1] (F2) {.};
        \node[vertex, below of = F2] (F3) {.};

        \draw[->-, looseness=30, out=135, in=45] (F1) to (F1);
        \draw[->-, looseness=30, out=135, in=45] (F2) to (F2);
        \draw[->-, bend left] (F1) to (F2);
        \draw[->-, bend left] (F2) to (F1);
        \draw[->-] (F3) to (F1);
        \draw[->-] (F3) to (F2);
        
        \node[right of = F2] (G) {$G:$};
        \node[vertex, right of = G] (G1) {.};
        \node[vertex, right of = G1] (G2) {.};
        \node[vertex, right of = G2] (G3) {.};

        \draw[->-, looseness=30, out=135, in=45] (G2) to (G2);
        \draw[->-, looseness=30, out=315, in=225] (G2) to (G2);
        \draw[->-, bend left] (G1) to (G2);
        \draw[->-, bend right] (G1) to (G2);
        \draw[->-, bend left] (G3) to (G2);
        \draw[->-, bend right] (G3) to (G2);
    \end{tikzpicture}
    \end{center}
\end{figure}
Then $E$ and $F$ are elementary balanced in-splits, and $F$ and $G$ are elementary balanced in-splits.
However, if \textbf{A} and \textbf{B} are the adjacency matrices of $E$ and $G$, respectively, then $\textbf{A}^2 \neq \textbf{B} \textbf{A}$, so $E$ and $G$ are not elementary balanced in-splits.
\end{example}

\begin{theorem}\label{thm:balanced-strong-shift-equivalence}
    Let $E$ and $F$ be finite graphs with no sinks.
    Then $E$ and $F$ are eventually conjugate if and only if there are graphs $E'$ and $F'$ which are conjugate to $E$ and $F$, respectively,
    whose adjacency matrices are balanced strong shift equivalent.
\end{theorem}

\begin{proof}
    If $E$ and $F$ are eventually conjugate, then there are graphs $E'$ and $F'$ which are conjugate to $E$ and $F$, respectively,
    such that $E'$ and $F'$ are balanced in-splits of a graph $G$.
    Then the adjacency matrices of $E'$ and $F'$ are balanced strong shift equivalent by Theorem~\ref{thm:generated} and Proposition~\ref{prop:balanced-elementary-eq}.

    For the other direction, it suffices to show that if the adjacency matrices \textbf{A} and \textbf{B} of $E$ and $F$, respectively, are balanced elementary equivalent,
    then the graphs are eventually conjugate.
    We use ideas of~\cite[Section 7.2]{LM}.
    Let $(\textbf{R}_E, \textbf{S}, \textbf{R}_F)$ be a balanced elementary equivalence from \textbf{A} to \textbf{B}.
    Then $\textbf{R}_E \textbf{S} = \textbf{R}_F$ is a nonnegative integer matrix so it defines a finite graph $G$.
    
    We construct a new graph as follows.
    Start by considering the disjoint union of the graphs $E$, $F$ and $G$.
    The edges in $E$ are called \emph{$E$-edges} and, similarly, the edges in $F$ and $G$ are called \emph{$F$-edges} and \emph{$G$-edges}, respectively.
    For each vertex $k\in G^0$ and $i\in E^0$ we add $\textbf{R}_E(k, i)$ edges from $k$ to $i$ called $\textbf{R}_E$-edges,
    and for each vertex $k\in G^0$ and $j\in F^0$ we add $\textbf{R}_F(k, j)$ edges from $k$ to $j$ called $\textbf{R}_F$-edges.
    Similarly, for vertices $i\in E^0$, $j\in F^0$ and $k\in G^0$, we add $\textbf{S}(i,k)$ edges from $i$ to $k$ and $\textbf{S}(j,k)$ edges from $j$ to $k$ called \textbf{S}-edges.
    This finishes the construction of the graph.
    An \textbf{S},$\textbf{R}_E$-path is an \textbf{S}-edge followed by an $\textbf{R}_E$-edge forming a path in the graph.
    We define $\textbf{S}, \textbf{R}_F$-paths, $\textbf{R}_E, \textbf{S}$-paths and $\textbf{R}_F, \textbf{S}$-paths analogously.

    Observe that we may identify the vertices of $E$ and $F$, and if $i\in E^0$ is paired with $j\in F^0$ then $\textbf{S}(i,k) = \textbf{S}(j,k)$ for all $k\in G^0$.
    Moreover, since $\textbf{A} = \textbf{S} \textbf{R}_E$, there is a bijection between $E$-edges and $\textbf{S}, \textbf{R}_E$-paths from vertices $i$ to $i'$ in $E^0$,
    and since $\textbf{B} = \textbf{S} \textbf{R}_F$ there is a bijection between $F$-edges and $\textbf{S}, \textbf{R}_E$-edges from vertices $j$ to $j'$ in $F^0$.
    Similarly, the relation $\textbf{R}_E \textbf{S} = \textbf{R}_F \textbf{S}$ defines a bijection between edges in $G$ and $\textbf{R}_E, \textbf{S}$-paths,
    and $\textbf{R}_E, \textbf{S}$-paths, respectively, from vertices $k$ to $k'$ in $G^0$.
    Fix such bijections.

    A $2$-path in $E$ now corresponds to an $\textbf{S}, \textbf{R}_E$-path followed by another $\textbf{S}, \textbf{R}_E$-path,
    and the middle two edges of these four edges are an $\textbf{R}_E, \textbf{S}$-path which in turn corresponds to a single edge in $G$.
    As explained in~\cite[Section 7.2]{LM} this defines a $2$-block map from $E$ to $G$.
    Similarly, there is a $2$-block map from $G$ to $F$.
    Extending the first map and composing it with the second, we obtain a map $E^3\LRA F^1$ which is compatible with $E$ and $F$.
    Moreover, if $x_0 x_1 x_2\in E^3$ is mapped to an edge $y\in F^1$ via this map and $i = s(x_0)$, 
    then the \textbf{S}-edge emitted from $i\in E^0$ (with some range $k\in G^0$) corresponds to a paired \textbf{S}-edge emitted from the vertex $j\in F^0$ paired with $i$ (also landing in $k\in G^0$).
    Since there is an $\textbf{R}_F$-edge emitted from $k\in G^0$ and landing in $s(y)\in F^0$, this defines an edge from $j$ to $s(y)$, say $y_0$.
    Therefore, we actually obtain a $(1,1)$-block map $\psi\colon E^3\LRA F^2$ sending $x_0 x_1 x_2\in E^3$ to $y_0 y\in F^2$.

    The block map $\psi$ extends to a $(1,1)$-sliding block code $h\colon E^\infty\LRA F^\infty$ on the path spaces.
    An inverse is constructed by applying the same procedure from $F$ to $E$, so $h$ is a homeomorphism.
    This shows that $E$ and $F$ are eventually conjugate.
\end{proof}


\begin{thebibliography}{99}
    \bibitem{Bates-Pask2004} T. Bates and D. Pask
        \emph{Flow equivalence of graph algebras},
        Ergodic Theory Dynam. Systems \textbf{24} (2004), 367--382.

    \bibitem{Boyle-Franks-Kitchens} M. Boyle, J. Franks and B. Kitchens,
        \emph{Automorphisms of one-sided subshifts of finite type},
        Ergodic Theory Dynam. Systems \textbf{10} (1990), no. 3, 421--449.

    \bibitem{BC} K.A. Brix and T.M. Carlsen,
        \emph{Cuntz--Krieger algebras and one-sided conjugacy of shift of finite type and their groupoids},
        J. Austral. Math. Soc., 1--10, doi:10.1017/S1446788719000168.

    \bibitem{BCW2017} N. Brownlowe, T.M. Carlsen and M. Whittaker,
        \emph{Graph algebras and orbit equivalence},
        Ergodic Theory Dynam. Systems \textbf{37} (2017), no. 2, 389--417.

    \bibitem{CEOR} T.M. Carlsen, S. Eilers, E. Ortega and G. Restorff,
        \emph{Flow equivalence and orbit equivalence for shifts of finite type and isomorphism of their groupoids},
        J. Math. Anal. Appl. \textbf{469} (2019), 1088--1110.

    \bibitem{Carlsen-Rout} T.M. Carlsen and J. Rout,
        \emph{Diagonal-preserving gauge-invariant isomorphisms of graph $\mathrm C^*$-algebras},
        J. Funct. Anal. \textbf{273} (2017), 2981--2993.

    \bibitem{CK80} J. Cuntz and W. Krieger,
        \emph{A class of {$\mathrm C^{\ast} $}-algebras and topological {M}arkov chains}, 
        Invent. Math. \textbf{56} (1980), 251--268.

    \bibitem{Eilers-Ruiz} S. Eilers and E. Ruiz,
        \emph{Refined moves for structure-preserving isomorphism of graph $\mathrm C^*$-algebras},
        arXiv:1908.03714v1 [math.OA], 46 pages.

    \bibitem{Kim-Roush} K.H Kim and F.W. Roush,
        \emph{The Williams conjecture is false for irreducible subshifts},
        Ann. of Math. (2) \textbf{149} (1999), 545--558.

    \bibitem{Kim-Roush1988} K.H. Kim and F.W. Roush,
        \emph{Decidability of shift equivalence},
        Dynamical systems (College Park, MD, 1986--87), 374--424, Lecture notes in Math., 1342, Springer, Berlin, 1988.

    \bibitem{Kitchens} B.P. Kitchens,
        \emph{Symbolic dynamics}, One-sided, two-sided and countable state Markov shifts,
        Universitext, Springer-Verlag, Berlin (1998), MR 1484730.

    \bibitem{LM} D. Lind and B. Marcus,
        \emph{An introduction to symbolic dynamics and coding},
        Cambridge University Press, Cambridge (1995), MR 1369092.

    \bibitem{Mat2017-circle} K. Matsumoto,
        \emph{Continuous orbit equivalence, flow equivalence of Markov shifts and circle actions on Cuntz-Krieger algebras},
        Math. Z. \textbf{285} (2017), 121–141.

    \bibitem{Mat17-UCOE} K. Matsumoto,
        \emph{Uniformly continuous orbit equivalence of {M}arkov shifts and gauge actions on {C}untz-{K}rieger algebras},
        Proc. Amer. Math. Soc. \textbf{145} (2017), 1131--1140.

    \bibitem{MM14} K. Matsumoto and H. Matui,
        \emph{Continuous orbit equivalence of topological Markov shifts and Cuntz--Krieger algebras},
        Kyoto J. Math. \textbf{54} (2014), 863--877.

    \bibitem{Parry-Sullivan} B. Parry and D. Sullivan,
        A topological invariant of flows on {$1$}-dimensional spaces,
        Topology \textbf{14} (1975), 297--299.

    \bibitem{Williams1973} R.F. Williams,
        \emph{Classification of subshifts of finite type},
        Ann.\ of Math. (2) \textbf{98} (1973), 120--153; errata, ibid. (2) \textbf{99} (1974), 380--381.
    \end{thebibliography}
\end{document}